\documentclass[12pt,twoside,leqno,a4paper]
{amsart}
\usepackage{amsmath} 
\usepackage{amssymb} 
\usepackage[english]{babel}
\usepackage{color}
\usepackage{enumitem}

\theoremstyle{plain}
\newtheorem{thm}{Theorem}[section]
\newtheorem{lem}[thm]{Lemma}
\newtheorem{defn}[thm]{Definition}
\newtheorem{prop}[thm]{Proposition}
\newtheorem{cor}[thm]{Corollary}
\newtheorem{rem}[thm]{Remark}

\newcommand{\C}{{\mathbb C}}

\newcommand{\R}{{\mathbb R}}

\newcommand{\N}{{\mathbb N}}
\newcommand{\e}{{\varepsilon}}

\newcommand{\Bn}{\mathbb{B}^n}

\renewcommand{\Re}{\operatorname{Re}}

\title[Hankel bilinear forms on  Fock-Sobolev spaces ]{Hankel bilinear forms on generalized Fock-Sobolev spaces on 
 $\C^n$}

\author{Carme Cascante}
\address{C. Cascante: Departament de Matem\`atiques i
    Inform\`atica, Universitat  de Barcelona,
     Gran Via 585, 08071 Barcelona, Spain}
\email{cascante@ub.edu}
\author{Joan F\`abrega}
\address{J. F\`abrega: Departament de Matem\`atiques i
    Inform\`atica, Universitat  de Barcelona,
     Gran Via 585, 08071 Barcelona, Spain}
\email{joan$_{-}$fabrega@ub.edu}
\author{Daniel Pascuas}
\address{D. Pascuas: Departament de Matem\`atiques i
    Inform\`atica, Universitat  de Barcelona,
     Gran Via 585, 08071 Barcelona, Spain}
\email{daniel$_{-}$pascuas@ub.edu}


\date{\today}

\thanks{The research 
	was supported in part by
        Ministerio de Econom\'{\i}a y Competitividad, 
Spain,   projects  
MTM2017-83499-P, MTM2017-90584-REDT 
and Generalitat de Catalunya, 
project
2017SGR358. The first author was also supported in 
part by Ministerio de Econom\'{\i}a y Competitividad, 
Spain,   project MDM-2014-0445}

\begin{document}

\begin{abstract}
We characterize the boundedness of 
 Hankel bilinear forms on a product of generalized 
Fock-Sobolev spaces on $\C^n$ 
with respect  to the weight 
$(1+|z|)^\rho e^{-\frac{\alpha}2|z|^{2\ell}}$, 
for $\ell\ge 1$, 
$\alpha>0$ and $\rho\in\R$.
We obtain  a weak decomposition of the Bergman kernel with estimates 
and a Littlewood-Paley formula, which are key ingredients in the proof of our main results.
As an application, we characterize  
the boundedness, compactness and the 
membership in the Schatten class of small Hankel 
operators on these spaces. 

\end{abstract}
\maketitle	

\section{Introduction} 
The main goal of this work is the characterization of the 
boundedness of  Hankel bilinear forms  on generalized 
Fock-Sobolev spaces.

Given a fixed number $\ell\ge 1$, 
 for $1\le p<\infty$,  $\alpha\ge 0$ and 
 $\rho\in\R$, we consider the space 
$L^{p,\ell}_{\alpha, \rho}:=L^{p,\ell}_{\alpha, \rho}(\C^n)$ 
 of all measurable functions $f$ on $\C^n$ 
 such that
	\[
	\|f\|^p_{L^{p,\ell}_{\alpha, \rho}}:=
	\int_{\C^n} \bigl|f(z)(1+|z|)^{\rho }
e^{-\frac \alpha 2 |z|^{2\ell}}\bigr|^p dV(z)
<\infty,
	\]
that is, $L^{p,\ell}_{\alpha, \rho}=L^p(\C^n;
(1+|z|)^{\rho p}e^{-\frac{\alpha p}2|z|^{2\ell}}dV(z))$.
Here $dV=dV_n$ denotes the Lebesgue measure on 
$\C^n$  normalized so that the measure of the unit ball 
$\Bn$  is 1. As usual, if $p=\infty$, 
 $L^{\infty,\ell}_{\alpha,\rho}
:=L^{\infty,\ell}_{\alpha,\rho}(\C^n)$
 consists of all measurable functions $f$ on $\C^n$
 such that
$	\displaystyle{
	\|f\|_{L^{\infty,\ell}_{\alpha,\rho}}:=
	\operatorname*{ess\,sup}_{z\in\C^n}|f(z)|
(1+|z|)^{\rho}e^{-\frac{\alpha}2|z|^{2\ell}}<\infty}.
$

For $\alpha>0$, we define the generalized Fock-Sobolev 
spaces  $F^{p,\ell}_{\alpha, \rho}:=H\cap L^{p,\ell}_{\alpha, \rho}$,
where  $H=H(\C^n)$ is the space of entire 
 functions on $\C^n$. We also  consider the little Fock space $\mathfrak{f}^{\infty,\ell}_{\alpha,\rho}$,
which is the closure of the space of  holomorphic polynomials in $F^{\infty,\ell}_{\alpha,\rho}$. Note that,
 for any $1 \le p<\infty$,  
 the holomorphic polynomials are also dense in  
$F^p_{\alpha,\rho}$
(see, for instance, \cite[Chapter 2]{Zhu2012} and 
Remark \ref{rem:density} below).

Since  $\ell\ge 1$ is fixed, 
from now on we will skip  it in our notations.
If $\rho=0$ we get the generalized Fock spaces
 $F^{p}_{\alpha}=F^{p}_{\alpha, 0}$, and we write
 $L^{p}_{\alpha}=L^{p}_{\alpha, 0}$.

Note that the space $L^{2}_\alpha$ is a Hilbert space with
 the inner product  given by the $\alpha$-pairing 
\begin{equation*}
\langle f,g\rangle_\alpha:=\int_{\C^n}
 f(z)\overline{g(z)}e^{-\alpha|z|^{2\ell}}dV(z),
\end{equation*}
and  $F^{2}_\alpha$ is a closed linear subspace of
 $L^{2}_\alpha$.

The Fock-Sobolev spaces $F^p_{\alpha,\rho}$ are the 
natural setting when we are dealing with Fock spaces. 
For instance, the pointwise estimate of a function in 
$F^p_{\alpha}$  as well as the norm estimates of the 
Bergman kernel are given in terms of weights 
corresponding to Fock-Sobolev spaces (see Corollary 
\ref{cor:pointwise} and Proposition 
\ref{prop:pqnormBergman}). Moreover, each derivative 
of a Fock function is in a Fock-Sobolev space 
(see  Theorem \ref{thm:LP}). 
So these spaces have been  subject of 
interest by several authors in recent years, specially for 
the case $\ell=1$ 
(see for instance \cite{Cho-Zhu}, 
\cite{Cho-Choe-Koo}, \cite{Cho-Isralowitz-Joo}, 
\cite{Mengestie} and the references therein).
As it happens for $\ell=1$ (see, for instance, \cite{isralowitz}, \cite{Ca-Fa-Pe} and the references therein), the model spaces $F^{p,\ell}_{\alpha,\varrho}$, $\ell>1$, should be useful to solve 
certain problems in  weighted Fock spaces 
$F^{p,\ell}_{\alpha}(\omega)$. 
This will be 
the object of forthcoming works.

We recall that a Hankel bilinear form  on a product of function spaces is a bilinear form $\Lambda$ satisfying 
$\Lambda(f,g)=\Lambda(fg,1)$.

Our first result characterizes the 
boundedness of the Hankel bilinear forms   on 
$F^p_{\alpha,\rho}\times F^{p'}_{\beta,\eta}$, where 
$p'=p/(p-1)$,  which extends the classical result in 
\cite{janson-peetre-rochberg} for $\ell=1$ and $\varrho=0$ (see also the recent paper \cite{Wang-Hu}).
In order to state our theorem, we consider 
 the space $E$  of entire functions of order $\ell$ and finite type, that is,
$
E=E(\C^n):=\{f\in H(\C^n):\,|f(z)|=O(e^{\tau |z|^\ell}), 
\,\text{for some }\,\tau>0\},
$
which is dense in $\mathfrak{f}^{\infty}_{\alpha,\rho}$ and  in $F^p_{\alpha,\rho}$, $1\le p<\infty$.

\begin{thm}\label{thm:bounded-bf}
Let $1\le p\le\infty$, $\alpha,\beta>0$ and 
 $\rho,\eta\in\R$.

A Hankel bilinear form $\Lambda:E\times E\to \C$ satisfies 
 $|\Lambda(f,g)|
\lesssim \|f\|_{F^p_{\alpha,\rho}}\|g\|_{F^{p'}_{\beta,\eta}}$
 if and only if 
there exists 
$b\in F^{\infty}_{\frac{\alpha+\beta}4,-\rho-\eta}$
 such that 
$\Lambda(f,g)
=\langle fg,b\rangle_{\frac{\alpha+\beta}2}$.
In this case, we have $\|\Lambda\|\simeq 
\|b\|_{F^{\infty}_{\frac{\alpha+\beta}4,-\rho-\eta}}$, 
and there exists $\varphi\in L^\infty_{0,-\rho-\eta}$ 
such that the bounded bilinear form 
$\widetilde{\Lambda}
:L^p_{\alpha,\rho}\times L^{p'}_{\beta,\eta}\to\C$,
defined by $\widetilde\Lambda(f,g)
=\langle fg,\varphi\rangle_{\frac{\alpha+\beta}2},$ 
coincides with $\Lambda$ on $E\times E$ and satisfies 
$\|\widetilde\Lambda\|\simeq\|\Lambda\|$.
\end{thm}

As a consequence,  we obtain a weak 
factorization of the space 
$F^1_{\alpha+\beta,\rho+\eta}$.
We recall that the weak product $F^{p}_{\alpha,\rho}\odot
F^{p'}_{\beta,\eta}$, $1\le p<\infty$, is the completion of the space of  finite sums 
$h=\sum_{j} f_j g_j$, $f_j\in F^{p}_{\alpha,\rho}$ and $g_j\in
F^{p'}_{\beta,\eta}$, using the norm
\[
\|h\|_{F^{p}_{\alpha,\rho}\odot F^{p'}_{\beta,\eta}}
:=\inf\left\{\sum_{j} 
\|f_j\|_{F^{p}_{\alpha,\rho}}\|g_j\|_{F^{p'}_{\beta,\eta}}: 
h=\sum_{j} f_jg_j\right\}.
\]
We then have:
\begin{cor}\label{cor:weakfact}
For $1\le p<\infty$,  $\alpha,\beta>0$ and 
 $\rho,\eta\in\R$, 
$F^{p}_{\alpha,\rho}\odot F^{p'}_{\beta,\eta}
=F^1_{\alpha+\beta,\rho+\eta}$.
Moreover, $F^{1}_{\alpha,\rho}\odot \mathfrak{f}^{\infty}_{\beta,\eta}
=F^{1}_{\alpha,\rho}\odot F^{\infty}_{\beta,\eta}
=F^1_{\alpha+\beta,\rho+\eta}$.
\end{cor}

Usually, necessary conditions for the 
boundedness of a bilinear form $\Lambda$ are obtained 
by checking the boundedness on  adequate 
testing functions $f$ and $g$. 
This is particularly simple when 
$\ell=1$, $\alpha=\beta$, $\rho=\eta=0$ and  $p=2$  
(see \cite{janson-peetre-rochberg}). 
In this classical case,  we can  take as  test functions $f$ and $g$ 
the square root of the Bergman kernel, that is 
$f(w)=g(w)=\sqrt{\alpha^n/n!\,} \,
e^{\frac {\alpha}2 z\overline{w}}$.
 Here, for $z,w\in\C^n$, 
$z\overline{w}:=\sum_{j=1}^n z_j\overline{w}_j$.
Then 
$
|b(z)|=|\langle fg,b\rangle_\alpha|
\le \|\Lambda\|\|f\|_{F^2_\alpha}^2
=\|\Lambda\|\, e^{\frac{\alpha}4|z|^2},
$
which proves that $b\in F^{\infty}_{\alpha/2}$.
Observe that the norm estimates of the above test functions $f$ and 
$g$ are similar to the ones of the Bergman kernel. 
This is not the situation in the general setting. 
In fact, although there is a broad literature 
on pointwise and 
norm estimates of the Bergman kernel for generalized 
Fock spaces (see, for instance, \cite{Delin}, 
\cite{Lindholm},  \cite{marzo-ortega-cerda}, 
\cite{seip-youssfi}, \cite{dallara}
and the references therein),
 it is not at all clear how to derive adequate 
decompositions of the Bergman kernel from these 
estimates.

For $\ell>1$ the choice of the test functions is  more delicate because the Bergman kernel
$K_\alpha(z,w)=\overline{K_{\alpha,z}(w)}$ is given 
in terms of derivatives of the so called Mittag-Leffler 
functions, which have zeros on $\C$
(see, for instance, Lemma \ref{lem:Bergmankernel} 
below and \cite[Theorem 2.1.1]{popov-sedletskii}).  
Consequently, it is not clear how to get 
a strong decomposition as in the previous case.
Instead, using the asymptotic behaviour  
 of the Mittag-Leffler functions, we obtain a weak 
decomposition of the Bergman kernel with accurate 
pointwise and norm estimates of each factor.
This will be a key tool to prove  Theorem 
\ref{thm:bounded-bf}.

\begin{thm}\label{thm:decompK}
Let $1\le p\le\infty$ $\alpha,\,\beta,\gamma>0$ 
and let $\rho, \eta\in\R$. Then there exist  functions 
$G_k=G_{k,\gamma,\alpha,\beta},\, 
H_k=H_{k,\gamma,\alpha,\beta}\in E(\C)$, 
$k=0,\cdots,n$,  such that:
\begin{align}
\label{eqn:decompK1}
K_\gamma(w,z)
&=\sum_{k=0}^n 
G_{k}(w\overline{z})H_{k}(w\overline{z}).\\
\label{eqn:decompK2}
\|K_{\gamma,z}\|_{F^1_{\alpha+\beta,\rho+\eta}}
&\simeq \sum_{k=0}^n
 \|G_{k}(\cdot\,\overline{z})\|
_{F^p_{\alpha,\rho}}
\|H_{k}(\cdot\,\overline{z})\|
_{F^{p'}_{\beta,\eta}}
\simeq (1+|z|)^{\rho+\eta} 
e^{\frac{\gamma^2\,|z|^{2\ell}}{2(\alpha+\beta)}}.
\end{align}
\end{thm}

If $\ell=1$, then 
$K_\gamma(w,z)=\frac{\gamma^n}{n!} 
e^{\gamma z\overline w}$ 
and in this case \eqref{eqn:decompK1}
reduces to 
\begin{equation}\label{eqn:kn1}
K_\gamma(w,z)=
\tfrac{\gamma^n}{n!} 
e^{\frac{\alpha\gamma}{\alpha+\beta} z\overline w}
\cdot e^{\frac{\beta\gamma}{\alpha+\beta} z\overline w}
=\,\tfrac {n!}{\gamma^n}
K_\gamma\left(w,\tfrac{\alpha}{\alpha+\beta}z\right)
K_\gamma\left(w,\tfrac{\beta}{\alpha+\beta}z\right).
\end{equation}

For  $\ell>1$, the explicit expression of the 
functions $G_k$ and $H_k$ is quite involved. 
A motivated definition of these factors as well as their  pointwise and norm estimates are given in Section \ref{sec:proof12} (see 
Definition \ref{defn:GHindecom} and 
Theorem \ref{thm:decompK2} below). 
In order to prove the norm estimates of the  
functions $G_k$ and $H_k$, we use, among other 
ingredients, the following  Littlewood-Paley type formula, 
which may be of independent interest by itself.

\begin{thm}\label{thm:LP}
Let  $1\le p\le\infty$, $\alpha>0$ and $\rho\in\R$.  
For an entire function $f$ in $\C^n$, let 
$|\nabla^m\,f|=\sum_{|\nu|=m}|\partial^\nu f|$, where
$|\nu|=\nu_1+\cdots+\nu_n$.

Then the following assertions are equivalent:
\begin{enumerate}
\item $f\in F^p_{\alpha,\rho}$.
\item For any $k\ge 1$, 
$|\nabla^k f|\in L^p_{\alpha,\rho-k(2\ell-1)}$.
\item For some $k\ge 1$, 
$|\nabla^k f|\in L^p_{\alpha,\rho-k(2\ell-1)}$.
\end{enumerate}
Moreover, we have
\begin{equation*}
\|f\|_{F^p_{\alpha,\rho}}
\simeq \sum_{m=0}^{k-1}|\nabla^m f(0)|
+\|\nabla^k f\|_{L^p_{\alpha,\rho-k(2\ell-1)}}.
\end{equation*}
\end{thm}

We point out that, in the particular case $\ell=1$, a fractional derivative  version of the Littlewood-Paley formula  is given in \cite{Cho-Choe-Koo}
(see also the references therein). 

Finally, as an application of Theorems  
\ref{thm:bounded-bf}  and \ref{thm:decompK}, 
we obtain a characterization of the boundedness, 
compactness and membership in the Schatten class 
of the small Hankel operators. 

\begin{thm}\label{thm:hankel} 
Let  $1\le p\le\infty$, $\alpha>0$ and $\rho\in\R$. For 
$\beta\in(\alpha,2\alpha)$ and $b\in F^\infty_\beta$, let
$\mathfrak{h}_{b,\alpha}$ be the small Hankel operator 
defined by $\mathfrak{h}_{b,\alpha}(f)
:=\overline{P_\alpha(\overline f\,b)}$, $f\in E$, where $P_\alpha$ is the Bergman projection on $F^2_{\alpha}$ (see Section \ref{sec:Bergman}). Then:
\begin{enumerate}
\item \label{item:hankelbound}
$\mathfrak{h}_{b,\alpha}$ extends to a bounded 
(compact) operator from $F^p_{\alpha,\rho}$ to 
$\overline{F^p_{\alpha,\rho}}$ if and only if $b\in 
F^\infty_{\frac{\alpha}2}$
 (respectively, $b\in \mathfrak{f}^\infty_{\frac{\alpha}2}$).
Moreover,  
$\|\mathfrak{h}_{b,\alpha}\|_{F^{p}_{\alpha,\rho}}
\simeq \|b\|_{F^\infty_{\frac{\alpha}2}}$.
\item \label{item:hankelschat}
$\mathfrak{h}_{b,\alpha}$ belongs to the Schatten class  
$S_p(F^2_{\alpha,\rho}, \overline{F^2_{\alpha,\rho}})$ if 
and only if $b\in F^p_{\frac{\alpha}2,\frac{2n(\ell-1)}p}$.
Moreover,
$\|\mathfrak{h}_{b,\alpha}\|_{S_p(F^2_{\alpha,\rho}, 
\overline{F^2_{\alpha,\rho}})}
\simeq \|b\|_{F^p_{\frac{\alpha}2,\frac{2n(\ell-1)}p}}$.
\end{enumerate}
\end{thm}

Unlike the case of small Hankel operators, 
there is a broad bibliography on the characterizations of  
boundedness, compactness and 
membership in the Schatten  class for Toeplitz operators   
on large families of weighted Fock 
spaces  (see, for instance, 
\cite{schuster-varolin}, \cite{hu-lv1},
\cite{isralowitz-virtanen-wolf}, 
\cite{oliver-pascuas}, \cite{hu-lv2} 
and the references therein). 
As far as we know, the literature on small Hankel 
operators is essentially concentrated around the case 
$\ell=1$. For instance, in the recent paper \cite{Wang-Hu} 
the authors characterize the boundedness and 
compactness of  small Hankel operators from 
$F^{p,1}_\alpha $ to $F^{q,1}_\alpha$, $0<p, q<\infty$ 
and $\alpha>0$.
Finally, we remark that for
$n=1$, $\varrho=0$ and $\ell$ is a 
positive integer, the boundedness, compactness and  
membership in the Hilbert-Schmidt class 
of the small Hankel operator were 
studied  in an unpublished manuscript  
written in collaboration with Jos\'e A. Pel\'aez 
\cite{Ca-Fa-Pa-Pe}.

The paper is organized as follows: 
In Section \ref{sec:Bergman} we state the main 
properties of the Fock-Sobolev spaces 
$F^{p}_{\alpha,\rho}$ 
and  the Bergman projection $P_\alpha$. 
The Littlewood-Paley formula  of 
Theorem \ref{thm:LP} and 
the weak factorization  of Theorem  
\ref{thm:decompK} will be  proved in 
Sections  \ref{sec:LP} and \ref{sec:proof12}, 
respectively. 
Section \ref{sec:proof11}
is devoted to the proof of Theorem \ref{thm:bounded-bf} and Corollary \ref{cor:weakfact}. 
Finally, in Section \ref{sec:proof13} we show  
Theorem \ref{thm:hankel}.

Throughout this paper the notation $\Phi\lesssim \Psi$ 
means that there exists a constant $C>0$, 
which does not depend on the involved variables,
 such that $\Phi\le C\, \Psi$. We write $\Phi\simeq \Psi$
 if $\Phi\lesssim \Psi$ and $\Psi\lesssim \Phi$.

\section{The Bergman projection on
 $L^{p}_{\alpha,\rho}$}\label{sec:Bergman}

In this section we state some well-known  properties of 
the Bergman projection and  the Fock-Sobolev spaces.

\subsection{On the two parametric Mittag-Leffler
 functions $E_{a,b}$}\quad\par

 The two parametric  Mittag-Leffler  functions are the
 entire functions on $\C$ given by 
\[
E_{a,b}(\lambda)
:=\sum_{k=0}^\infty \frac{\lambda^k}{\Gamma(a k+b)}
\qquad(\lambda\in\C,\,\, a,b>0).
\]
Observe that  $E_{1,1}(\lambda)$ is just the exponential 
function $e^\lambda$.

A good general reference for the  Mittag-Leffler
 functions is the  
book~{\cite{gorenflo-kilbas-mainardi-rogosin}}.

In this section we recall the asymptotic expansions of 
the two parametric Mittag-Leffler functions and their
 derivatives.
Those expansions will be useful to obtain both pointwise
 and norm estimates of the Bergman kernel.

\begin{thm}[{\cite[Theorem 1.2.1]{popov-sedletskii}}]
	Let $a\in (0,1]$ and let $b>0$. 
Then, for $|\lambda|\to\infty$, we have
\begin{equation}\label{eqn:Eab}
E_{a,b}(\lambda)=
\begin{cases}	
\frac 1a \lambda^{(1-b)/a}e^{\lambda^{1/a}}
+O(\lambda^{-1}),
	 & \quad\text{if}\quad|\arg\lambda|\le \frac{7\pi}8\,a,\\
O(\lambda^{-1}),
	&  \quad\text{if}\quad|\arg\lambda|\ge \frac{5\pi}8\,a.
	\end{cases}
	\end{equation}
\end{thm}

Here, for $\lambda\in\C\setminus\{0\}$, 
 $\arg\lambda$ denotes  the principal branch of the 
 argument of $\lambda$,  that is, 
 $-\pi<\arg\lambda\le\pi$. 
Moreover, for $\beta\in\R$,
 $\lambda^\beta=|\lambda|^\beta
e^{i\beta\arg\lambda}$.

By Cauchy's formula (see \cite[Theorem 1.4.2]{olver}) 
we can differentiate the asymptotic expansion  
\eqref{eqn:Eab} on smaller sectors to obtain:

\begin{cor}
	Let $a\in (0,1]$, $b>0$ and  $m\in\N$. 
Then, for $|\lambda|\to\infty$, we have that
\[
E^{(m)}_{a,b}(\lambda)=
\begin{cases}	
\frac{\lambda^{(m(1-a)+1-b)/a}}{a^{m+1}}
e^{\lambda^{1/a}}
(1+O(\lambda^{-1/a}))+O(\lambda^{-m-1}),
 & \text{if}\,\,|\arg \lambda|\le \frac{3\pi}{4}a,\\
O(\lambda^{-m-1}),
 & \text{if}\,\,|\arg \lambda|\ge \frac{3\pi}{4}a.
\end{cases}
\] 
\end{cor}

From this result we deduce pointwise estimates 
of the function $E^{(m)}_{a,b}$. 
In order to state these estimates we introduce the 
following function.
\begin{defn}
For $ c\ge 0$, let
\begin{equation}\label{eqn:varphi}
\varphi_{c}(\lambda):=\begin{cases}
|e^{c\lambda^\ell}|, 
&\text{if $	|\arg\lambda|\le\frac{\pi}{2\ell}$},\\
1, &\text{otherwise}.
\end{cases}
\end{equation}
\end{defn}

\begin{cor}\label{cor:estimatesEab}
If  $b\in (0,1]$ and $m\in\N$, then
\begin{equation}\label{eqn:estimatesEab}
|E^{(m)}_{\frac 1\ell,b}(\lambda)|\lesssim 
(1+|\lambda|)^{m(\ell-1)+(1-b)\ell}\varphi_{1}(\lambda).
\end{equation}
\end{cor}

\subsection{The Bergman projection}\quad\par

We denote by $P_\alpha$ the  Bergman projection from
 $L^{2}_\alpha$ onto $F^{2}_\alpha$ defined by 
\[
P_\alpha(f)(z)=
\langle f,K_{\alpha,z}\rangle_\alpha
=\int_{\C^n}f(w)K_\alpha(z,w)e^{-\alpha |w|^{2\ell}}dV(w),
\]
where $K_\alpha$ is the Bergman kernel and 
$K_{\alpha,z}(w):=\overline{K_\alpha(z,w)}
=K_\alpha(w,z)$.

The first result in this section states that the Bergman 
kernel can be described in terms of derivatives of the 
Mittag-Leffler function $E_{1/\ell,1/\ell}$. 
In order to do that, we recall some standard notations.  $\N$ will denote the set of non-negative entire
 numbers. For a multi-index 
$\nu=(\nu_1,\cdots,\nu_n)\in \N^n$
  and $z=(z_1,\cdots,z_n)\in \C^n$,  we use the standard 
 notations   $z^\nu=z_1^{\nu_1}\cdots z_n^{\nu_n}$,
$\nu!=\nu_1!\cdots\nu_n!$ and 
 $|\nu|=\nu_1+\cdots+\nu_n$. We then have (see, for 
instance, \cite[\S 5]{bommier-englis-youssfi}):

\begin{lem}
\label{lem:Bergmankernel}
The system 
$\bigl\{\frac{w^\nu}{\|w^\nu\|_{F^{2}_\alpha}}\bigr\}_{\nu\in
\N^n}$ is an orthonormal basis for  $F^{2}_\alpha$, 
so the Bergman kernel is
\[
	K_\alpha(z,w)=\overline{K_{\alpha,z}(w)}
	=\sum_{\nu\in \N^n}\frac{z^\nu\overline{w}^\nu}
{\|w^\nu\|^2_{F^{2}_{\alpha}}}.
\]
Namely, since 
	$\|w^\nu\|^2_{F^{2}_{\alpha}}=
	\frac {\alpha^{-\frac{|\nu|+n}\ell}}{\ell}\frac{n!\, \nu!\, 
\Gamma\left(\frac{|\nu|+n}\ell\right)}{(n-1+|\nu|)!}
	$,
$
K_{\alpha}(z,w)=H_{\alpha}(z\overline{w})
$, where
\[
H_{\alpha}(\lambda)
:=\frac{\ell\alpha^{n/\ell}}{n!}
\sum_{k=0}^\infty 
\frac{(n-1+k)!}{k!}\frac{\alpha^{k/\ell}\lambda^k}
{\Gamma\left(\frac{k+n}\ell\right)}
=\frac{\ell\alpha^{n/\ell}}{n!}
E_{1/\ell,1/\ell}^{(n-1)}(\alpha^{1/\ell}\lambda).
\]	
In particular, for any $\delta>0$ we have 
\begin{equation}\label{eqn:Kdelta}
K_{\alpha}(z,\delta w)
=K_{\alpha}(\delta z,w)
=\delta^{-n}K_{\alpha\delta^\ell }(z,w).
\end{equation}
\end{lem}

As a consequence of \eqref{eqn:estimatesEab} and the 
fact that the Taylor coefficients of the function 
$E_{\frac 1\ell,\frac 1\ell}$ are positive, we obtain the 
following pointwise estimate of the Bergman kernel.

\begin{prop}\label{prop:estimatesK}
For $\alpha>0$ we have
\[
|K_\alpha(z,w)|\lesssim 
(1+|z|)^{n(\ell-1)}(1+|w|)^{n(\ell-1)}
\varphi_\alpha(z\overline w).
\]
In particular, if $|z|\le M$ then 
\[
|K_\alpha(z,w)|\lesssim (1+|w|)^{n(\ell-1)}
\,e^{\alpha M^\ell|w|^\ell}
\lesssim e^{\alpha(M+1)^\ell|w|^\ell},
\]
so $K_\alpha(\cdot,z)\in E$, for every $z\in\C^n$.
\end{prop}

The next results will be used to prove our main theorems.

\begin{prop}\label{prop:pqnormBergman}
	Let $1\le p\le\infty$, $\alpha,\gamma>0$ and $\rho\in\R$. 
Then
	\[
	\|K_{\gamma}(\cdot,z)\|_{F^{p}_{\alpha,\rho}}
	\simeq (1+|z|)^{\rho+2n(\ell-1)/p'}
	e^{\frac{\gamma^2}{2\alpha}|z|^{2\ell}}\quad(z\in\C^n).
	\]
\end{prop}

The proof of Proposition \ref{prop:pqnormBergman} for  
$\rho=0$ is in \cite{bommier-englis-youssfi}, while the 
general case can be found in 
\cite[Corollary 2.11]{Ca-Fa-Pa}.

\begin{prop}
\label{prop:Ponto}
	Let  $1\le p\le \infty$ and $\rho\in\R$.
	If  $0\le \alpha<2\gamma$ then the Bergman projection
	 $P_\gamma$ is bounded from  $L^{p}_{\alpha,\rho}$ 
	 onto $F^{p}_{\gamma^2/(2\gamma-\alpha),\rho}$.
Moreover,  $P_\gamma$ is the identity operator on  
$F^{p}_{\alpha,\rho}$.
In particular,  
$P_\alpha:L^{p}_{\alpha,\rho}\to F^{p}_{\alpha,\rho}$ 
is bounded.
\end{prop}

The condition $\alpha<2\gamma$ 
ensures that the projection is well defined, in the sense 
that if $\varphi\in L^{p}_{\alpha,\rho}$ then 
$\varphi\,K_{\gamma,z}\in L^1_{2\gamma}$.

The proof of this proposition when  $\rho=0$  can be 
found in \cite{janson-peetre-rochberg} ($\ell=1$)
 and in \cite{bommier-englis-youssfi} ($\ell>1$).
The general case can be found in 
\cite[Proposition 4.2]{Ca-Fa-Pa}.

Observe that  by Proposition \ref{prop:Ponto} 
$f=P_\alpha(f)$, for any $f\in F^p_{\alpha,\rho}$.
Hence H\"older's inequality and Proposition 
\ref{prop:pqnormBergman} give the following 
elementary pointwise estimate.

\begin{cor} 
\label{cor:pointwise}
Let $1\le p\le\infty$, $\alpha>0$ and $\rho\in\R$. Then 
\[
|f(z)|\lesssim 
\|f\|_{F^p_{\alpha,\rho}}(1+|z|)^{-\rho+2n(\ell-1)/p}
e^{\frac{\alpha}2|z|^{2\ell}}
\quad (f\in F^p_{\alpha,\rho},\,z\in\C^n),
\]
and so 
$F^p_{\alpha,\rho}\hookrightarrow  F^\infty_{\alpha,\rho-2n(\ell-1)/p}$.
\end{cor}

Using Corollary \ref{cor:pointwise} and simple 
pointwise estimates of the weights, it is easy to prove 
the following result. A detailed proof can be found in 
{\cite{Ca-Fa-Pa}}, where  we give a complete 
characterization of the embbedings 
$F^{p}_{\alpha,\rho}\hookrightarrow F^{q}_{\beta,\eta}$.

\begin{cor}
\label{cor:embedF}
Let $1\le p,q\le\infty$, $\alpha>0$ and $\rho,\eta\in\R$. 

\begin{enumerate}
\item If $\beta>\alpha$, then 
$
F^{p}_{\alpha,\rho}\hookrightarrow  F^{q}_{\beta,\eta}
$
and 
$F^{p}_{\alpha,\rho}\hookrightarrow 
\mathfrak{f}^{\infty}_{\beta,\eta}$.
\item
If  $\rho+2n(\ell-1)/p'\le \eta$ then 
$F^{1}_{\alpha,\eta}\hookrightarrow 
F^p_{\alpha,\rho}$.
\end{enumerate}
\end{cor}

The next interpolation result will be used in the 
forthcoming sections 
(see, for instance, \cite[Lemma 3.10]{Ca-Fa-Pa}).
 
\begin{lem}
\label{lem:interpolation}
Let $1<p<\infty$. Then for $\theta=1/p'$ we have 
\[
(F^1_{\alpha,\rho},F^\infty_{\alpha,\rho})_{[\theta]}
=F^p_{\alpha,\rho}.
\]
\end{lem}

Next lemma studies 
 the action of dilations on Fock-Sobolev spaces.

\begin{lem}\label{lem:pairing}
Let $1\le p\le\infty$, $\alpha, \beta>0$ and $\rho\in\R$.  
For $\delta>0$ we have:
\begin{enumerate}
\item \label{item:pairing1} The dilation operator 
$f\mapsto f(\delta\cdot)$ is a topological  isomorphism 
from $F^p_{\alpha,\rho}$ onto 
$F^p_{\delta^{2\ell}\alpha,\rho}$.
\item \label{item:pairing2} If $f,g\in E$, then $\langle 
f,g\rangle_\alpha=\delta^{2n}\langle 
f(\cdot),g(\delta^2\cdot)\rangle_{\delta^{2\ell}\alpha}$.
\item \label{item:pairing3}  If $f\in E$, $g\in 
F^p_{\beta,\rho}$ and 
$\delta^{2\ell}<2\alpha/\beta$, then 
$f(\cdot)\,g(\delta^2\cdot)\in L^1_{2\delta^{2\ell}\alpha}$. 
\end{enumerate}
\end{lem}
\begin{proof}
The change of variables $w=\delta z$  easily gives 
\eqref{item:pairing1}. The same change of variables 
together with  the orthogonality of the monomials give 
\eqref{item:pairing2}, since 
$
\langle f,g\rangle_\alpha=\delta^{2n}\langle 
f(\delta\cdot),g(\delta\cdot)\rangle_{\delta^{2\ell}\alpha}=
\delta^{2n}\langle 
f,g(\delta^2\cdot)\rangle_{\delta^{2\ell}\alpha}.
$
Finally, assertion \eqref{item:pairing3} follows from 
\eqref{item:pairing1} and  \eqref{item:pairing2}.
\end{proof}

\begin{rem}\label{rem:density}
As it happens in the classical case $\ell=1$ and $\rho=0$ 
(see, for instance, \cite[Proposition 2.9]{Zhu2012}),  
Lemma \ref{lem:pairing} \eqref{item:pairing1} and 
Corollary \ref{cor:embedF} allow us to prove the 
density of the holomorphic polynomials in 
$F^p_{\alpha,\rho}$, $1\le p<\infty$. Indeed,
if $f\in F^p_{\alpha,\rho}$ and  $f_\delta:=f(\delta\cdot)$, 
$0<\delta<1$, then 
$f_\delta\in F^p_{\delta^{2\ell}\alpha,\rho}\subset 
F^2_{\delta^{\ell}\alpha,\rho}\subset F^p_{\alpha,\rho}$.
Now, standard arguments give 
$\|f_\delta-f\|_{F^p_{\alpha,\rho}}\to 0$ as $\delta\to1^-$.
Finally, for fixed $0<\delta<1$ there is a 
sequence of polynomials $\{q_{\delta,k}\}_k$ such that 
$\|f_\delta-q_{\delta,k}\|_{F^2_{\delta^\ell\alpha,\rho}}
\to 0$ as $k\to\infty$, so 
$\|f_\delta-q_{\delta,k}\|_{F^p_{\alpha,\rho}}\to 0$.
\end{rem}

We finish this section with a duality result that  we 
will use later. 
Its proof is standard, but since
 we have not found an explicit reference, 
for a sake of completeness we supply a sketch of the  proof.

\begin{prop}\label{prop:dualF}
If $1\le p<\infty$ and $\alpha/2\le\gamma<2\alpha$, then the dual 
$(F^p_{\alpha,\rho})'$ of $F^p_{\alpha,\rho}$
(with respect to the $\gamma$-pairing)  is 
$F^{p'}_{\frac{\gamma^2}\alpha,-\rho}$. 
Moreover, the   dual of 
$\mathfrak{f}^\infty_{\alpha,\rho}$ is 
$F^{1}_{\frac{\gamma^2}\alpha,-\rho}$.
\end{prop}

\begin{proof}
First we prove that if $g\in F^{p'}_{\frac{\gamma^2}\alpha,-\rho}$, then $f\in E\to \langle f,g\rangle_\gamma$ extends to  a bounded linear form on $F^p_{\alpha,\rho}$.
Since $0<\alpha\le 2\gamma$, Proposition \ref{prop:Ponto} gives $F^{p'}_{\frac{\gamma^2}\alpha,-\rho}
=P_\gamma(L^{p'}_{2\gamma-\alpha,-\rho})$. Therefore, if  $g\in F^{p'}_{\frac{\gamma^2}\alpha,-\rho}$, then there exists $\varphi\in L^{p'}_{2\gamma-\alpha,-\rho}$ such that  $g=P_\gamma(\varphi)$ and  $\|\varphi\|_{L^{p'}_{2\gamma-\alpha}}
\simeq \|g\|_{F^{p'}_{\frac{\gamma^2}\alpha,-\rho}}
$.
As a consequence, 
\[
|\langle f,g\rangle_{\gamma}|
=|\langle f,\varphi\rangle_{\gamma}
\le \|\varphi\|_{L^{p'}_{2\gamma-\alpha}}
\|f\|_{F^p_{\alpha,\rho}}
\simeq \|g\|_{F^{p'}_{\frac{\gamma^2}\alpha,-\rho}}
\|f\|_{F^p_{\alpha,\rho}} \quad (f\in E).
\]
In order to prove the converse, observe that 
Lemma \ref{lem:pairing}\eqref{item:pairing1}  with 
$\delta^{2\ell}=\gamma/\alpha$ reduces the proof to 
the case $\gamma=\alpha$. 
Namely,  $b\in F^{p'}_{\alpha,-\rho}$ if and only if 
$g=b(\delta^2\cdot)\in 
F^{p'}_{\frac{\gamma^2}\alpha,-\rho}$, and, 
since by hypothesis $\gamma^2/\alpha<2\gamma$, we 
have that for any  $f\in E$  $fg\in L^1_{2\gamma}$ and 
$\langle f,b\rangle_\alpha=\langle f,g\rangle_\gamma$.
	From the classical $L^p$-duality it is easy to check 
that  the dual of $L^{p}_{\alpha,\rho}$ with respect to 
the $\alpha$-pairing is 
$L^{p'}_{\alpha,-\rho}$.
This result together with Proposition \ref{prop:Ponto},
for $\alpha=\beta$, prove the duality for 
$F^{p}_{\alpha,\rho}$.

   Next we deal with the duality of 
$\mathfrak{f}^\infty_{\alpha,\rho}$. 
Note that if $b\in F^{1}_{\alpha,-\rho}$ then
    $\langle\cdot,b\rangle_\alpha\in  
(\mathfrak{f}^{\infty}_{\alpha,\rho})^*$
    and
  $\|\langle\cdot,b\rangle_\alpha\|
_{(\mathfrak{f}^{\infty}_{\alpha,\rho})^*}
    \lesssim \|b \|_{F^{1}_{\alpha,-\rho}}$.

Conversely, given 
$u\in(\mathfrak{f}^{\infty}_{\alpha,\rho})^*$, we are 
going to prove that there is  $b\in F^{1}_{\alpha,-\rho}$ 
such that $u=\langle\cdot,b\rangle_\alpha$ and 
$\|b\|_{F^{1}_{\alpha,-\rho}}\lesssim\|u\|_{(\mathfrak{f}^{
\infty}_{\alpha,\rho})^*}$.
Choose $\alpha/2<\beta<\alpha$. By Corollary 
\ref{cor:embedF} we have 
$F^{2}_{\beta}\hookrightarrow 
\mathfrak{f}^{\infty}_{\alpha,\rho}$ and so the
restriction of $u$ to $F^{2}_{\beta}$ is a bounded 
linear form on this space.
It follows that there is $g\in F^{2}_{\beta}$ such that
$u(f)=\langle f,g\rangle_\beta$, for every $f\in E$. 
Now, 
by Lemma \ref{lem:pairing} with  
 $\delta^{2\ell}=\frac\alpha\beta<2$, we have  
$b=g(\delta^2\cdot)\in F^2_{\delta^{4\ell}\beta}
=F^2_{\frac{\alpha^2}\beta}$ 
and 
$
u(f)=
\langle f,\,b\rangle_{\alpha}$, for any $f\in E$.

Thus it only remains to prove that 
$\|b\|_{L^{1}_{\alpha,-\rho}}
\lesssim\|u\|_{(\mathfrak{f}^{\infty}_{\alpha,\rho})^*}$.

For $f\in C_c(\C^n)$, let  
$Tf(z):=f(z)(1+|z|)^{-\rho}e^{\frac{\alpha}2|z|^{2\ell}}
\in L^\infty_{\alpha,\rho}$. 
Then we have $\|P_\alpha(Tf)\|_{ F^\infty_{\alpha,\rho}}
\lesssim \|Tf\|_{ L^\infty_{\alpha,\rho}}=\|f\|_{L^\infty}$. 
Since $f$ is compactly supported, Proposition 
\ref{prop:estimatesK} gives that  $P_\alpha(Tf)\in E$. 
Then, by duality,
\[
\|b\|_{L^{1}_{\alpha,-\rho}}
=\sup_{\substack{ f\in C_c(\C^n)\\\|f\|_{L^\infty}=1}}
|\langle Tf,b\rangle_\alpha|
=
\sup_{\substack{ f\in C_c(\C^n)\\\|f\|_{L^\infty}=1}}
|u(P_{\alpha}(T f))|\lesssim 
\|u\|_{(\mathfrak{f}^{\infty}_\alpha)^*}.
\qedhere
\]
\end{proof}

\section{Proof of  Theorem \ref{thm:LP}}
\label{sec:LP}

We begin the section with the following technical lemma.

\begin{lem}\label{lem:Phi}
 For $c\in \R$, 
 let $\Phi_{c,z}(w):=\varphi_c(w\overline z)$,
where $\varphi_c$ is defined by \eqref{eqn:varphi}.
Then, for any $1\le p\le \infty$, $\alpha>0$, $\rho\in\R$ 
and $c\in [0,\alpha]$,
\[
\|\Phi_{c,z}\|_{L^{p}_{\alpha,\rho}}
\simeq (1+c^{1/\ell} |z|)^{\rho-2n(\ell-1)/p}
e^{\frac{c^2}{2\alpha}|z|^{2\ell}}.
\]
\end{lem}

\begin{proof}
Let $1\le p<\infty$. Given $z\in\C^n$, pick an unitary 
mapping $U_z$ on $\C^n$ which maps $z$ to 
$(|z|,0)\in\C\times\C^{n-1}$. Then making the change 
of variables $v=U_z w$ and integrating in polar 
coordinates (see {\cite[Lemma 2.9]{Ca-Fa-Pa}}
for a detailed proof of the second equivalence) we get  
\begin{align*}
\|\Phi_{c,z}\|_{L^{p}_{\alpha,\rho}}^p
&\simeq
\int_\C\varphi_c(|z|v_1)^p
\int_{\C^{n-1}}(1+|v_1|+|v'|)^{\rho p}\,
e^{-\frac{\alpha p}2(|v_1|^2+|v'|^2)^\ell} dV(v') dA(v_1)\\
&\simeq
\int_\C\varphi_c(|z|u)^p
(1+|u|)^{\rho p-2(n-1)(\ell-1)}\,
e^{-\frac{\alpha p}2|u|^{2\ell}} dA(u)\\
&=
\int_{\{|u|\ge 1, |\arg u|\le\frac \pi{2\ell}\}}
(1+|u|)^{\rho p-2(n-1)(\ell-1)}\,
\bigl| e^{c|z|^\ell u^\ell-\frac{\alpha }2|u|^{2\ell}}
\bigr|^p dA(u)\\
&\quad+
\int_{\{|u|< 1, |\arg u|\le\frac \pi{2\ell}\}}
(1+|u|)^{\rho p-2(n-1)(\ell-1)}\,\bigl| 
e^{c|z|^\ell u^\ell-\frac{\alpha }2|u|^{2\ell}}\bigr|^pdA(u)\\
&\quad +
\int_{\{|\arg u|>\frac \pi{2\ell}\}}
(1+|u|)^{\rho p-2(n-1)(\ell-1)}\,
e^{-\frac{\alpha p}2|u|^{2\ell}}dA(u)=:I_1+I_2+I_3.
\end{align*}
For $|\arg u|\le\frac{\pi}{2\ell}$, we have  
$
\Re(c u^\ell|z|^\ell)-\tfrac{\alpha}{2}|u|^{2\ell}=
\tfrac{c^2}{2\alpha}|z|^{2\ell}
-\tfrac{\alpha}2\bigl|\tfrac{c}{\alpha}|z|^\ell-u^\ell\bigr|^2.
$
Hence, the change $\lambda=u^\ell$ gives
\begin{align*}
I_1&=e^{\tfrac{c^2p}{2\alpha}|z|^{2\ell}}
\int_{\{|u|\ge 1, |\arg u|\le\frac \pi{2\ell}\}}
(1+|u|)^{\rho p-2(n-1)(\ell-1)}\, 
e^{-\tfrac{\alpha p}2\bigl|
\tfrac{c}{\alpha}|z|^\ell-u^\ell\bigr|^2}dA(u)\\
&\lesssim e^{\tfrac{c^2p}{2\alpha}|z|^{2\ell}}
\int_{\C}
(1+|\lambda|)^{\tfrac{\rho p-2n(\ell-1)}{\ell}}\,
 e^{-\tfrac{\alpha p}2
\bigl|\tfrac{c}{\alpha}|z|^\ell-\lambda\bigr|^2}dA(\lambda)\\
&\lesssim  (1+ |z|)^{\rho p-2n(\ell-1)}
e^{\frac{c^2 p}{2\alpha}|z|^{2\ell}}.
\end{align*}   
The proof of the last inequality for $|z|\le 1$ is clear. For $|z|>1$, splitting the integral over $\C$ as a sum of the integral on the set 
\[
A=\bigl\{\lambda\in\C: \tfrac{c}{2\alpha}|z|^\ell\le \bigl|\lambda\bigr|
\le\tfrac{2c}{\alpha}|z|^\ell\bigr\} 
\] 
and the integral on $\C\setminus A$, it is easy to check that  
$I_1\lesssim (1+|z|)^{\rho p-2n(\ell-1)}+e^{-\e |z|^{2\ell}}$
for some $\e>0$, which proves the result (see \cite[Lemma 2.10]{Ca-Fa-Pa} for more details).

The estimates of $I_2$ and $I_3$ are much easier. 
Clearly $I_3\lesssim 1$ and, since $|e^{c|z|^\ell u^\ell}|
\le e^{c|z|^\ell}$, for $|u|<1$, we also have 
$I_2\lesssim e^{cp|z|^\ell}$, 
which completes the case $p<\infty$.

Next assume $p=\infty$. 
In this case, arguing as above,
\begin{align*}
\|\Phi_{c,z}\|_{L^{\infty}_{\alpha,\rho}}
&\simeq
\sup_{v_1\in\C}\varphi_c(|z|v_1)\sup_{v'\in\C^{n-1}}
(1+|v_1|+|v'|)^{\rho }\,
e^{-\frac{\alpha }2(|v_1|^2+|v'|^2)^\ell}
\end{align*}
It is easy to check that  
\[
\sup_{v'\in\C^{n-1}}
(1+|v_1|+|v'|)^{\rho }\,
e^{-\frac{\alpha }2(|v_1|^2+|v'|^2)^\ell}
\simeq (1+|v_1|)^{\rho }\,e^{-\frac{\alpha }2|v_1|^{2\ell}},
\]
so $\|\Phi_{c,z}\|_{L^{\infty}_{\alpha,\rho}}
\simeq M_{c}(z)+L_{c}(z)$, 
where
\begin{align*}
M_{c}(z)&=\sup_{|\arg u|\le\frac{\pi}{2\ell}}\varphi_c(|z|u)
(1+|u|)^{\rho }\,e^{-\frac{\alpha }2|u|^{2\ell}},\\
L_{c}(z)&=\sup_{|\arg u|>\frac{\pi}{2\ell}}\varphi_c(|z|u)
(1+|u|)^{\rho }\,e^{-\frac{\alpha }2|u|^{2\ell}}.
\end{align*}

Now
\begin{align*}
M_c(z)
&\simeq e^{\tfrac{c^2}{2\alpha}|z|^{2\ell}} 
\sup_{|\arg \lambda|\le\frac{\pi}{2}}
(1+|\lambda|)^{\rho/\ell}\,
 e^{-\tfrac{\alpha}2\bigl|\tfrac{c}{\alpha}|z|^\ell-\lambda
\bigr|^2}\\
&= e^{\tfrac{c^2}{2\alpha}|z|^{2\ell}} 
\sup_{r>0}
(1+r)^{\rho/\ell}\,
 e^{-\tfrac{\alpha}2\bigl|\tfrac{c}{\alpha}|z|^\ell-r
\bigr|^2}.
\end{align*}
It is easy to check that the last supremum is equivalent
 to $(1+c^{1/\ell}|z|)^\rho$ 
(see for instance {\cite[Lemma 2.8]{Ca-Fa-Pa}}). 
Moreover, $L_c(z)\simeq 1$. Hence
\[
\|\Phi_{c,z}\|_{L^{\infty}_{\alpha,\rho}}\simeq  
(1+c^{1/\ell}|z|)^\rho e^{\tfrac{c^2}{2\alpha}|z|^{2\ell}}
 +1 \simeq (1+c^{1/\ell}|z|)^\rho 
 e^{\tfrac{c^2}{2\alpha}|z|^{2\ell}},
\]
which ends the proof.
\end{proof}

\begin{proof}[Proof of Theorem \ref{thm:LP}]
The proof of the theorem is a consequence of the  
following assertions:
\begin{enumerate}
\item[1)]The linear  operators 
$f\mapsto \partial_{z_j}f$ are bounded from 
$F^p_{\alpha,\rho}$ 
to $F^p_{\alpha,\rho+1-2\ell}$.
\item[2)] The linear operators 
\[
S_j(g)(z):=z_j \int_0^1g(tz)\, dt,\quad j=1,\cdots,n,
\]
are bounded from $F^p_{\alpha,\rho+1-2\ell}$ to 
$F^p_{\alpha,\rho}$.
\end{enumerate}
Taking for granted these results it is easy to prove the 
case  $k=1$. Indeed, assertion 1) shows that if 
$f\in F^{p}_{\alpha,\rho}$, then 
$|\nabla f|\in L^{p}_{\alpha,\rho+1-2\ell}$. 
Moreover,   the identity 
\[
f(z)=f(0)+\sum_{j=1}^n \int_0^1  z_j \partial_{z_j}f (tz)\, dt,\qquad f\in H(\C^n),
\]
together with assertion 2) give the converse.
Combining these results we have
\[
\|f\|_{F^p_{\alpha,\rho}}\simeq |f(0)|+ 
\|\nabla f\|_{L^p_{\alpha,\rho+1-2\ell}}
\simeq  |f(0)|
+\sum_{j=1}^n \|\partial_{z_j} f\|_{F^p_{\alpha,\rho+1-2\ell}}.
\]
Iterating this argument we prove the general case.

Next, we prove  the two assertions. 
We begin showing that the linear  operator 
$f\mapsto \partial_{z_j}f$ is bounded from 
$F^p_{\alpha,\rho}$ 
to $F^p_{\alpha,\rho+1-2\ell}$.
By interpolation 
(see Lemma \ref{lem:interpolation}) it is sufficient to 
prove this result 
for $p=1$ and $p=\infty$.

By Proposition  \ref{prop:Ponto}, $f=P_\alpha(f)$, so 
\[
\partial_{z_j}f(z)=\int_{\C^n}f(w)\,
\partial_{z_j}K_{\alpha}(z,w)\,e^{-\alpha|w|^{2\ell}}dV(w).
\]
Therefore Lemma \ref{lem:Bergmankernel} and 
Corollary \ref{cor:estimatesEab} imply 
\[
|\partial_{z_j}K_{\alpha}(z,w)|
\simeq\left|
\overline{w}_j\,E_{1/\ell,1/\ell}^{(n)}
(\alpha^{1/\ell}z\overline{w})\right|
\lesssim |w|(1+|z\overline w|)^{(n+1)(\ell-1)}
\varphi_\alpha(z\overline{w}),
\]
where $\varphi_\alpha$ is defined by \eqref{eqn:varphi}.
 Hence 
\[
|\partial_{z_j}f(z)|\lesssim 
\int_{\C^n}|f(w)| T_{\alpha}(z,w) 
 e^{-\alpha|w|^{2\ell}}dV(w),
\]
where 
$T_\alpha(z,w)
:=(1+|z|)^{(n+1)(\ell-1)}(1+|w|)^{(n+1)(\ell-1)+1}
\varphi_{\alpha}(z\overline{w})
$.
Thus
\begin{align*}
\|\partial_{z_j}f\|_{F^1_{\alpha,\rho-2\ell+1}}
&\lesssim \int_{\C^n} |f(w)| 
\|T_\alpha(\cdot,w)\|_{F^1_{\alpha,\rho-2\ell+1}} 
e^{-\alpha|w|^{2\ell}} dV(w),\\
\|\partial_{z_j}f\|_{F^\infty_{\alpha,\rho-2\ell+1}}
&\lesssim \|f\|_{F^\infty_{\alpha,\rho}}\sup_{z\in\C^n} 
(1+|z|)^{\rho-2\ell+1}e^{-\frac\alpha 2|z|^{2\ell}}
\|T_\alpha(z,\cdot)\|_{L^1_{\alpha,-\rho}}.
\end{align*}
 Now  Lemma \ref{lem:Phi} shows 
\begin{align*}
\|T_\alpha(\cdot,w)\|_{L^1_{\alpha,\rho-2\ell+1}}=
(1+|w|)^{n(\ell-1)+\ell}
\|\Phi_{\alpha,w}\|_{{L^1_{\alpha,\rho+n(\ell-1)-\ell}}}
\simeq (1+|w|)^{\rho}e^{\frac\alpha 2|w|^{2\ell}} 
\end{align*}
and
\begin{align*}
\|T_\alpha(z,\cdot)\|_{L^1_{\alpha,-\rho}}=
(1+|z|)^{(n+1)(\ell-1)}
\|\Phi_{\alpha,z}\|_{{L^1_{\alpha,-\rho+n(\ell-1)+\ell}}}
\simeq (1+|z|)^{-\rho+2\ell-1}e^{\frac\alpha 2|z|^{2\ell}}.
\end{align*}
Hence
$\|\partial_{z_j}f\|_{F^p_{\alpha,\rho-2\ell+1}}
\lesssim \|f\|_{F^p_{\alpha,\rho}}$ 
for $p=1$ and $p=\infty$.
Consequently, for any $p$, we have 
\[
 |f(0)|+ 
\|\nabla f\|_{L^p_{\alpha,\rho+1-2\ell}}\lesssim 
\|f\|_{F^p_{\alpha,\rho}}.
\]

To complete the proof we show that the operators $S_j$ 
map $F^p_{\alpha,\rho+1-2\ell}$ to $F^p_{\alpha,\rho}$, 
$p=1$ and $p=\infty$. 
Indeed, 
 Proposition  \ref{prop:Ponto} gives
\[
S_j(g)(z)=
z_j\int_0^1\int_{\C^n}g
(w)K_\alpha(tz,w)e^{-\alpha|w|^{2\ell}}\,dV(w)\, dt.
\]
Therefore Proposition \ref{prop:pqnormBergman} gives
\begin{align*}
\|S_j(g)\|_{F^1_{\alpha,\rho}}&\le
\int_0^1\int_{\C^n}|g(w)| 
\|K_\alpha(\cdot,tw)\|_{F^1_{\alpha,\rho+1}} 
e^{-\alpha|w|^{2\ell}}\,dV(w)\, dt\\
&\lesssim \int_{\C^n}|g(w)|  e^{-\alpha|w|^{2\ell}}\,
\int_0^1(1+t|w|)^{\rho+1}\,e^{\frac{\alpha}2(t|w|)^{2\ell}}\,
dt\,dV(w),
\end{align*}
and
\begin{align*}
\|S_j(g)\|_{F^\infty_{\alpha,\rho}}&\lesssim
\|g\|_{L^\infty_{\alpha,\rho+1-2\ell}}\sup_z 
(1+|z|)^{\rho+1}e^{-\frac{\alpha}2|z|^{2\ell}}
\int_0^1 \|K_\alpha(\cdot,tz)\|_{F^1_{\alpha,2\ell-\rho-1}} \, 
dt\\
&\lesssim \|g\|_{L^\infty_{\alpha,\rho+1-2\ell}}\sup_z 
(1+|z|)^{\rho+1}e^{-\frac{\alpha}2|z|^{2\ell}}
\int_0^1(1+t|z|)^{2\ell-\rho-1}\,
e^{\frac{\alpha}{2}(t|z|)^{2\ell}}\,dt.
\end{align*}
Therefore, the norm estimates  
$\|S_j(g)\|_{F^1_{\alpha,\rho}}
\lesssim \|g\|_{F^1_{\alpha,\rho+1-2\ell}}$
 and $\|S_j(g)\|_{F^\infty_{\alpha,\rho}}
\lesssim \|g\|_{F^\infty_{\alpha,\rho+1-2\ell}}$ 
follow from 
\[
\int_0^1\,(1+ta)^\tau \,e^{(ta)^{2\ell}}\, dt\le c_\tau
(1+a)^{\tau-2\ell} \,e^{a^{2\ell}} \quad(a>0),
\]
which can be easily derived by splitting the integral as a 
sum of the integrals from $0$ to $1/2$ and from $1/2$ 
and $1$.

Altogether gives that 
\[\|f\|_{F^p_{\alpha,\rho}}\lesssim
 |f(0)|+ 
\|\nabla f\|_{L^p_{\alpha,\rho+1-2\ell}},
\]
which ends the proof of Theorem \ref{thm:LP}.
\end{proof}

\vspace{.5cm}

As a consequence, we deduce the following result that will be used in the next section.

\begin{cor}
\label{cor:diffzw}
Let $f\in E(\C)$ and let 
$k=0,1,\cdots$. Then:
\begin{enumerate}
\item \label{item:diffzw1}
There exists $\tau=\tau(k)>0$ such that 
$|f^{(k)}(w \overline z)|
=O\bigl(e^{\tau|z|^\ell|w|^\ell}\bigr)$.
In particular,  $f^{(k)}(\cdot\,\overline z)\in E(\C^n)$, 
for every $z\in\C^n$.
\item \label{item:diffzw2}
For any $1\le p\le\infty$ we have 
\[
\|f^{(k)}(\cdot\,\overline z)\|_{F^{p}_{\alpha,\rho-k(2\ell-1)}}
\lesssim 
 (1+|z|)^{-k}\|f(\cdot\,\overline z)\|_{F^{p}_{\alpha,\rho}}
\qquad(|z|\ge 1).
\]
\end{enumerate}
\end{cor}

\begin{proof}
Let us begin by observing that if $f\in E(\C)$, then 
$f^{(k)}\in E(\C)$. This proves \eqref{item:diffzw1}.

For $z\ne 0$, pick a unitary mapping $U_z$ which
 maps $z$ to $(|z|,0)\in\C\times\C^{n-1}$. 
Then making the change of variables $v=U_z w$ and
 defining $g_z(v_1,v')=f(|z|v_1)$, Theorem \ref{thm:LP}
 gives
\begin{align*} 
\|f^{(k)}(\cdot\,\overline z)\|_{F^{p}_{\alpha,\rho-k(2\ell-1)}}
&
=|z|^{-k} \left\|\tfrac{\partial^{k}g_z}{\partial v_1^k}
\right\|_{F^{p}_{\alpha,\rho-k(2\ell-1)}}\\
&\lesssim|z|^{-k} \left\|g_z\right\|_{F^{p}_{\alpha,\rho}}
=|z|^{-k}\left\|f(\cdot\,\overline{z})\right\|_{F^{p}_{\alpha,\rho}}. 
 \qedhere
\end{align*}
\end{proof}

\section{Proof of Theorem \ref{thm:decompK}}
\label{sec:proof12}

From the asymptotic expansion \eqref{eqn:Eab}
it is easy to check the following result. 
\begin{lem} 
 For $0<\theta<1$ there exists $R_{\ell,\theta}\in H(\C)$ 
such that 
\begin{equation}\label{eqn:factorizationE}
E_{1/\ell,1/\ell}(\lambda)=c_{\ell,\theta}
E_{\frac 1\ell,\frac{\ell+1}{2\ell}}(\theta^{1/\ell}\lambda)\,
E_{\frac1\ell,\frac{\ell+1}{2\ell}}
((1-\theta)^{1/\ell}\lambda)\,
+R_{\ell,\theta}(\lambda),
\end{equation}
where $c_{\ell,\theta}=
\frac{(\theta(1-\theta))^{\frac{1-\ell}{2\ell}}}{\ell}$.
Moreover, by \eqref{eqn:estimatesEab},
\begin{align}
\label{eqn:factorizationE1}
|E_{\frac 1\ell,\frac{\ell+1}{2\ell}} 
(\theta^{1/\ell}\lambda)|
&\lesssim (1+|\lambda|)^{\frac{\ell-1}{2}}
\varphi_\theta(\lambda),\\
\label{eqn:factorizationE2}
|E_{\frac 1\ell,\frac{\ell+1}{2\ell}} 
((1-\theta)^{1/\ell}\lambda)|
&\lesssim (1+|\lambda|)^{\frac{\ell-1}{2}} 
\varphi_{1-\theta}(\lambda),\\
\label{eqn:factorizationR}
|R_{\ell,\theta}(\lambda)|
&\lesssim (1+|\lambda|)^{\frac{\ell-3}{2}}
(\varphi_\theta(\lambda)+\varphi_{1-\theta}(\lambda),
\end{align}
where $\varphi_{c}$ is the function defined by 
 \eqref{eqn:varphi}.
\end{lem}

For $\ell=1$ the identity \eqref{eqn:factorizationE} 
reduces to 
$e^\lambda=e^{\theta\lambda}e^{(1-\theta)\lambda}$ 
and $R_{1,\theta}=0$.

\begin{cor}\label{cor:factorizationK}
Let  $\theta\in (0,1)$ and let $\tilde\theta=1-\theta$. Then
\begin{align*}
K_\gamma(z,w)&=C_{\ell,\gamma,\theta}
\sum_{k=0}^{n-1}\mbox{$\binom{n-1}{k}$}
\theta^{\frac k\ell}
 E^{(k)}_{\frac{1}{\ell},\frac{\ell+1}{2\ell}}
(\left(\theta\gamma\right)^{1/\ell}z\overline{w})
\tilde\theta^{\frac{n-1-k}\ell}
E^{(n-1-k)}_{\frac{1}{\ell},\frac{\ell+1}{2\ell}}
(\tilde\theta\gamma)^{\frac 1\ell}z\overline{w})\\
&+\frac{\ell\gamma^{n/\ell}}{n!}
R^{(n-1)}_{\ell,\theta}(\gamma^{\frac 1\ell}z\overline{w}),
\end{align*}
where $C_{\ell,\gamma,\theta}
=\frac{\ell\gamma^{n/\ell}c_{\ell,\theta}}{n!}$.
\end{cor}
Observe that if $\ell=1$, and 
$\theta=\frac{\alpha}{\alpha+\beta}$ this decomposition 
is just  \eqref{eqn:kn1}.

In order to prove \eqref{eqn:decompK1} 
we introduce the following definitions.
For  $k=0,\cdots,n-1$, let
\begin{align}
\label{eqn:GH1}
G_{k,\gamma,\theta}(\lambda)
&:=\mbox{$\binom{n-1}{k}$}\theta^{\frac k\ell} 
\frac{\ell\gamma^{n/\ell}c_{\ell,\theta}}{n!}
 E^{(k)}_{\frac{1}{\ell},\frac{\ell+1}{2\ell}}
(\left(\theta\gamma\right)^{1/\ell}\lambda),\\
\label{eqn:GH2}
H_{k,\gamma,\theta}(\lambda)
&:=(1-\theta)^{\frac{n-1-k}\ell}
 E^{(n-1-k)}_{\frac{1}{\ell},\frac{\ell+1}{2\ell}}
(\left((1-\theta)\gamma\right)^{1/\ell}\lambda),\\
\label{eqn:R2}
R_{n,\gamma,\theta}(\lambda)
&:=\frac{\ell\gamma^{n/\ell}}{n!}R^{(n-1)}_{\ell,\theta}
(\gamma^{\frac 1\ell}\lambda).
\end{align}

We claim that:
\begin{prop}\label{prop:claims}
Let $\gamma>0$ and $\theta\in (0,1)$. Then, 
for any $z\in\C^n$, the functions 
$G_{k,\gamma,\theta}(\cdot\,\overline{z}),
\,H_{k,\gamma,\theta}(\cdot\,\overline{z})$, 
$k=0,\cdots,n-1$, and 
$R_{n,\gamma,\theta}(\cdot\,\overline{z})$
belong to $E$. Moreover, for  $1\le p\le\infty$, $\alpha>0$ 
and $\rho\in\R$ we have 
\begin{align}
\label{claim:E1}
& \|G_{k,\gamma,\theta}(\cdot\,\overline{z})\|
_{F^p_{\alpha,\rho}}
\lesssim 
(1+|z|)^{\rho+(\ell-1)(2k+1-2n/p)}
e^{\frac{\theta^2\gamma^2}{2\alpha }|z|^{2\ell}},
\,k=0,\cdots,n-1,\\
\label{claim:R1}
&\|R_{n,\gamma,\theta}(\cdot\,\overline{z})\|
_{F^p_{\alpha,\rho}}
\lesssim 
(1+|z|)^{\rho+(\ell-1)(2n/p'-1)}\bigl(
e^{\frac{\theta^2\gamma^2}{2\alpha }|z|^{2\ell}}+
e^{\frac{(1-\theta)^2\gamma^2}{2\alpha }|z|^{2\ell}}
\bigr).
\end{align}
\end{prop}

 Observe that replacing $G_{k,\gamma,\theta}$ by 
$H_{k,\gamma,\theta}$ and $p,\alpha,\rho,\theta, k$ by 
$p',\beta,\eta,1-\theta,n-1-k$, respectively, we obtain
\begin{align}
\label{claim:E2}
 \|H_{k,\gamma,\theta}(\cdot\,\overline{z})\|
_{F^{p'}_{\beta,\eta}}
&\lesssim 
(1+|z|)^{\eta+(\ell-1)(2(n-1-k)+1-2n/{p'})}
e^{\frac{(1-\theta)^2\gamma^2}{2\beta }|z|^{2\ell}},\\
\label{claim:R2}
\|R_{n,\gamma,\theta}(\cdot\,\overline{z})\|
_{F^{p'}_{\beta,\eta}}
&\lesssim 
(1+|z|)^{\eta+(\ell-1)(2n/p-1)}\bigl(
e^{\frac{\theta^2\gamma^2}{2\beta }|z|^{2\ell}}+
e^{\frac{(1-\theta)^2\gamma^2}{2\beta}|z|^{2\ell}}
\bigr).
\end{align}

Taking for granted these estimates, we conclude the 
proof of Theorem \ref{thm:decompK}.

We first state the following definition.

\begin{defn}\label{defn:GHindecom}
For $\alpha, \beta, \gamma>0$ and $k=0,\cdots,n-1$, we define the following entire functions on $\C$ given by:
\begin{align*}
&G_k(\lambda)
:=G_{k,\gamma,\frac{\alpha}{\alpha+\beta}}(\lambda), 
&&H_k(\lambda)
:=H_{k,\gamma,\frac{\alpha}{\alpha+\beta}}(\lambda),
&&&
\\
&G_{n}(\lambda)
:=R_{n,\gamma,\frac{\alpha}{\alpha+\beta}}(\lambda), 
&& H_{n}(\lambda):=1, 
&&& \text{if $\alpha\ge\beta$,}\\
&G_{n}(\lambda):=1,
&&H_{n}(\lambda)
:=R_{n,\gamma,\frac{\alpha}{\alpha+\beta}}(\lambda), 
&&&\text{if $\alpha<\beta$.}
\end{align*}
\end{defn}

\begin{proof}[Proof of Theorem \ref{thm:decompK}]
By Corollary \ref{cor:factorizationK}, it is clear that the 
functions $G_k$ and $H_k$ in Definition \ref{defn:GHindecom} satisfy equation
\eqref{eqn:decompK1}. 

Next we prove \eqref{eqn:decompK2}.
By Proposition \ref{prop:pqnormBergman} and 
H\"older's inequality we have
\[
(1+|z|)^{\rho+\eta} 
e^{\frac{\gamma^2}{2(\alpha+\beta)}|z|^{2\ell}}
\simeq
\|K_\gamma(\cdot,z)\|_{F^1_{\alpha+\beta,\rho+\eta}}
\lesssim \sum_{k=0}^n
 \|G_{k,\gamma,\theta}(\cdot\overline{z})\|
_{F^p_{\alpha,\rho}}
\|H_{k,\gamma,\theta}(\cdot\overline{z})\|
_{F^{p'}_{\beta,\eta}}.
\]
By \eqref{claim:E1} and \eqref{claim:E2}, 
\[
\sum_{k=0}^{n-1}
 \|G_{k,\gamma,\theta}(\cdot\,\overline{z})\|
_{F^p_{\alpha,\rho}}
\|H_{k,\gamma,\theta}(\cdot\,\overline{z})\|
_{F^{p'}_{\beta,\eta}}
\lesssim (1+|z|)^{\rho+\eta}
e^{\frac{\gamma^2}{2 } \psi(\theta)|z|^{2\ell}},
\]
where 
 $\psi(\theta)=\frac{\theta^2}{\alpha }+
\frac{(1-\theta)^2}{\beta }$. 
Since  
$\psi(\theta)\ge\psi\bigl(\frac{\alpha}{\alpha+\beta}\bigr)
=\frac1{\alpha+\beta}$,
\[
\sum_{k=0}^{n-1}
 \|G_{k}(\cdot\,\overline{z})\|_{F^p_{\alpha,\rho}}
\|H_{k}(\cdot\,\overline{z})\|_{F^{p'}_{\beta,\eta}}
\lesssim (1+|z|)^{\rho+\eta}
e^{\frac{\gamma^2}{2(\alpha+\beta) }|z|^{2\ell}}.
\]

Assume $\alpha\ge \beta$. Now  
\eqref{claim:R1}, with 
$\theta=\frac{\alpha}{\alpha+\beta}$,
 shows that 
\begin{align*}
 \|G_{n}(\cdot\,\overline{z})\|_{F^p_{\alpha,\rho}}&
\|H_{n}(\cdot\,\overline{z})\|_{F^{p'}_{\beta,\eta}}\\
&\lesssim 
(1+|z|)^{\rho+(\ell-1)(2n/p'-1)}\bigl(
e^{\frac{\alpha\gamma^2}{2(\alpha +\beta)^2}|z|^{2\ell}}+
e^{\frac{\beta^2\gamma^2}
{2\alpha(\alpha +\beta)^2 }|z|^{2\ell}}
\bigr)\\
&\lesssim (1+|z|)^{\rho+\eta}
e^{\frac{\gamma^2}{2(\alpha+\beta) }|z|^{2\ell}}.
\end{align*}
By using \eqref{claim:R2} we obtain the same estimate  
for $\alpha<\beta$.
\end{proof}

For further references, we consider convenient to state 
 the following more precise  version of  
Theorem \ref{thm:decompK}, which provides an explicit 
decomposition of the Bergman kernel $K_\gamma$ with 
norm-estimates of the factors.

\begin{thm}\label{thm:decompK2}
Let $1\le p\le\infty$ $\alpha,\,\beta,\gamma>0$ 
and let $\rho, \eta\in\R$. Then, 
\begin{equation}
\label{eqn:decompK21}
K_\gamma(w,z)
=\sum_{k=0}^{n-1}
G_{k,\gamma,\frac{\alpha}{\alpha+\beta}}(w\overline{z})
H_{k,\gamma,\frac{\alpha}{\alpha+\beta}}(w\overline{z})
+ 
R_{n,\gamma,\frac{\alpha}{\alpha+\beta}}(w\overline{z}),
\end{equation}
and, for $k=0,\cdots,n-1$,
\begin{align}
\label{claim:G21}
 &\|G_{k,\gamma,\frac{\alpha}{\alpha+\beta}}
(\cdot\,\overline{z})\|_{F^p_{\alpha,\rho}}
\lesssim 
(1+|z|)^{\rho+(\ell-1)(2k+1-2n/p)}
e^{\frac{\alpha\gamma^2}{2(\alpha+\beta)^2 }|z|^{2\ell}},
\\
\label{claim:H21}
 & 
\|H_{k,\gamma,\frac{\alpha}{\alpha+\beta}}
(\cdot\,\overline{z})\|_{F^{p'}_{\beta,\eta}}
\lesssim 
(1+|z|)^{\eta+(\ell-1)(2n/p-2k-1)}
e^{\frac{\beta\gamma^2}{2(\alpha+\beta)^2 }|z|^{2\ell}},
\\
\label{claim:R21}
 &
\|R_{n,\gamma,\frac{\alpha}{\alpha+\beta}}
(\cdot\,\overline{z})\|_{F^p_{\alpha,\rho}}
\lesssim 
(1+|z|)^{\rho+(\ell-1)(2n/p'-1)}
e^{\frac{\alpha\gamma^2}{2(\alpha+\beta)^2} |z|^{2\ell}}
,\quad\text{if $\alpha\ge \beta$},\\
\label{claim:R22}
 &
\|R_{n,\gamma,\frac{\alpha}{\alpha+\beta}}
(\cdot\,\overline{z})\|_{F^{p'}_{\beta,\eta}}
\lesssim 
(1+|z|)^{\eta+(\ell-1)(2n/p-1)}
e^{\frac{\beta\gamma^2}{2(\alpha+\beta)^2 }|z|^{2\ell}}
,\quad\text{if $\alpha< \beta$}.
\end{align}
Therefore, defining
\begin{align*} 
&G_{n,\gamma,\frac{\alpha}{\alpha+\beta}}
=R_{n,\gamma,\frac{\alpha}{\alpha+\beta}},\quad 
\text{and}\quad 
H_{n,\gamma,\frac{\alpha}{\alpha+\beta}}=1,
\quad && \text{if $\alpha\ge \beta$},\\
&G_{n,\gamma,\frac{\alpha}{\alpha+\beta}}=1,\quad
\text{and}\quad
H_{n,\gamma,\frac{\alpha}{\alpha+\beta}}
=R_{n,\gamma,\frac{\alpha}{\alpha+\beta}},
\quad && \text{if $\alpha< \beta$},
\end{align*}
we obtain
\begin{align*}
\|K_{\gamma,z}\|_{F^1_{\alpha+\beta,\rho+\eta}}
&\simeq \sum_{k=0}^{n}
 \|G_{k,\gamma,\frac{\alpha}{\alpha+\beta}}
(\cdot\,\overline{z})\|_{F^p_{\alpha,\rho}}
\|H_{k,\gamma,\frac{\alpha}{\alpha+\beta}}
(\cdot\,\overline{z})\|_{F^{p'}_{\beta,\eta}}\\
&\simeq(1+|z|)^{\rho+\eta} 
e^{\frac{\gamma^2\,|z|^{2\ell}}{2(\alpha+\beta)}}.
\end{align*}
\end{thm}

Next we prove Proposition \ref{prop:claims}.

\begin{proof}[Proof of Proposition \ref{prop:claims}] 
In order to simplify the notations, for $k=0,\cdots,n-1$ 
we write $G_k$ and 
$H_k$ instead of $G_{k,\gamma,\theta}$ and 
$H_{k,\gamma,\theta}$, respectively, and $R_n$ 
instead of $R_{n,\gamma,\theta}$.

By Corollary \ref{cor:estimatesEab} all the Mittag-Leffler 
functions in the identity \eqref{eqn:factorizationE} are in 
$E(\C)$, so $R_{\ell,\theta}\in E(\C)$. Therefore, 
Corollary \ref{cor:diffzw}\eqref{item:diffzw1} shows 
that  $G_k(\cdot\,\overline{z})$, 
$H_k(\cdot\,\overline{z})$ and 
$R_n(\cdot\,\overline{z})$ are in $E(\C^n)$.

Next, we prove \eqref{claim:E1} and \eqref{claim:R1}.
Since $G_k$ and $R_n$ are in $E$,  there exists $\tau>0$ 
such that, for every $|z|\le 1$, 
$|G_k(\cdot\,\overline{z})|,
\,|R_n(\cdot\,\overline{z})|
\lesssim e^{\tau|w|^\ell}$ and consequently 
$\|G_k(\cdot\overline{z})\|_{F^p_{\alpha,\rho}},
\,\|R_n(\cdot\overline{z})\|_{F^p_{\alpha,\rho}}
\lesssim 1$.

Next consider $|z|> 1$. 
 The estimate \eqref{claim:E1} follows from 
Corollary \ref{cor:diffzw}\eqref{item:diffzw2}, 
\eqref{eqn:factorizationE1} 
and Lemma \ref{lem:Phi}. Indeed,
\begin{align*}
\|E^{(k)}_{\frac 1\ell,\frac{\ell+1}{2\ell}} 
(\cdot\,(\theta\gamma)^{1/\ell}\overline{z})\|
_{F^p_{\alpha,\rho}}
&\lesssim (1+|z|)^{-k}
\|E_{\frac 1\ell,\frac{\ell+1}{2\ell}} 
(\cdot\,(\theta\gamma)^{1/\ell}\overline{z})\|
_{F^p_{\alpha,\rho+k(2\ell-1)}}\\
&\lesssim 
(1+|z|)^{\frac{\ell-1}2-k} \|\Phi_{\theta\gamma,z}\|
_{L^p_{\alpha,\rho+\frac{\ell-1}2+k(2\ell-1)}}\\
&\lesssim 
(1+|z|)^{\rho+(\ell-1)(2k+1-2n/p)}
e^{\frac{\theta^2\gamma^2}{2\alpha }|z|^{2\ell}}. 
\end{align*}

In order to  prove \eqref{claim:R1} for $|z|>1$, we follow 
the same arguments used to prove  \eqref{claim:E1}. 
Note that  \eqref{eqn:factorizationR} shows that 
$R_{\ell,\theta}$satisfies an estimate similar to the one 
satisfied by $E_{\frac 1\ell,\frac{\ell+1}{2\ell}}$:
\begin{equation}\label{eqn:Restimate}
|R_{\ell,\theta}(\lambda)|\lesssim 
(1+|\lambda|)^{\frac{\ell-1}{2}}
(\varphi_\theta(\lambda)+\varphi_{1-\theta}(\lambda)).
\end{equation}
Then
 Corollary \ref{cor:diffzw}\eqref{item:diffzw2}, 
 \eqref{eqn:Restimate} and Lemma \ref{lem:Phi} 
give
\begin{align*}
\|R^{(n-1)}_{\ell,\theta} 
(\cdot\,\gamma^{\frac 1\ell}\overline{z})\|
_{F^p_{\alpha,\rho}}
&\lesssim (1+|z|)^{1-n}
\|R_{\ell,\theta} 
(\cdot\,\gamma^{\frac 1\ell}\overline{z})\|
_{F^p_{\alpha,\rho+(n-1)(2\ell-1)}}\\
&\lesssim 
(1+|z|)^{\frac{\ell-1}2+1-n}
\|\Phi_{\theta\gamma,z}\|_{L^p_{\alpha,\rho+\frac{\ell-1}2+
(n-1)(2\ell-1)}}\\
&\quad+(1+|z|)^{\frac{\ell-1}2+1-n}
\|\Phi_{(1-\theta)\gamma,z}\|
_{L^p_{\alpha,\rho+\frac{\ell-1}2+(n-1)(2\ell-1)}}
\\
&\lesssim 
(1+|z|)^{\rho+(\ell-1)(2n/p'-1)}\bigl(
e^{\frac{\theta^2\gamma^2}{2\alpha }|z|^{2\ell}}+
e^{\frac{(1-\theta)^2\gamma^2}{2\alpha }|z|^{2\ell}}
\bigr). \qedhere
\end{align*}
\end{proof}

\begin{rem}
 If $\ell $ is a positive integer, then $e^{\lambda^\ell}$ 
is a zero-free  entire function. 
Therefore, we have the strong decomposition
\[
K_\gamma(w,z)=\left[
e^{\frac{\alpha\gamma}{\alpha+\beta} (z\overline w)^\ell}
\right]
\cdot \left[e^{-\frac{\alpha\gamma}{\alpha+\beta} (z\overline w)^\ell}\, K_\gamma(w,z)\right],
\]
whose terms can be estimated with the same methods used above (for $n=1$, $\alpha=\beta$ and $\varrho=\eta=0$, see \cite{Ca-Fa-Pa-Pe}.)
\end{rem}

\section{Proof of  Theorem \ref{thm:bounded-bf} and 
Corollary \ref{cor:weakfact}}
\label{sec:proof11}

\subsection{Proof of Theorem \ref{thm:bounded-bf}}
Assume that $|\Lambda(f,g)|\lesssim 
\|f\|_{F^p_{\alpha,\rho}}\|f\|_{F^{p'}_{\beta,\eta}}$, for 
$f,g\in E$.
The first observation is that if there exists $\gamma>0$, 
$\tau\in\R$ and $b\in F^{\infty}_{\gamma,\tau}$ such that 
$\Lambda(f,g)=\langle fg,b\rangle_\gamma$, 
for $f,g\in E$, then Proposition \ref{prop:Ponto}  and 
Theorem \ref{thm:decompK} give
\begin{align*}
|b(z)|=|\langle K_\gamma(\cdot,z),b\rangle_\gamma|
&\le \sum_{k=0}^n|\Lambda(G_k(\cdot\,\overline{z}),
H_k(\cdot\,\overline{z}))|\\
&\lesssim \|\Lambda\| (1+|z|)^{\rho+\eta} 
e^{\frac{\gamma^2}{2(\alpha+\beta)}|z|^{2\ell}}.
\end{align*}
Thus $b\in F^{\infty}
_{\frac{\gamma^2}{\alpha+\beta},-\rho-\eta}$ and 
$\|b\|_{ F^{\infty}
_{\frac{\gamma^2}{\alpha+\beta},-\rho-\eta}}\lesssim 
 \|\Lambda\|$.

Therefore it is enough to prove that there exists 
$b\in F^\infty_{\frac{\alpha+\beta}2,\tau}$
such that 
$\Lambda(f,g)
=\langle fg,b\rangle_{\frac{\alpha+\beta}2}$, 
for every $f,g\in E$.

Let $E^p_{\alpha,\rho}=(E,\|\cdot\|_{F^p_{\alpha,\rho}})$
and assume $\alpha\ge\beta$.  
The boundedness of the bilinear form $\Lambda$ on 
$E^p_{\alpha,\rho}\times E^{p'}_{\beta,\eta}$ implies that 
$f\mapsto \Lambda(f,1)$ is a bounded linear form on 
$E^p_{\alpha,\rho}$.
Since  $\frac{\alpha+\beta}2\le \alpha$, 
Corollary  \ref{cor:embedF} shows 
that $E^{1}_{\frac{\alpha+\beta}2,-\tau}\hookrightarrow 
E^p_{\alpha,\rho}$, for any 
$\tau\le -\rho-2n(\ell-1)/p'$.
In particular, $f\mapsto \Lambda(f,1)$ is a bounded 
linear form on  $E^{1}_{\frac{\alpha+\beta}2,-\tau}$.
Therefore, using  that 
$\Lambda$ is a Hankel form and the fact that de dual of 
 $E^{1}_{\frac{\alpha+\beta}2,-\tau}$ with respect to 
the $\frac{\alpha+\beta}2$-pairing is 
$F^{\infty}_{\frac{\alpha+\beta}2,\tau}$ (see Proposition  
\ref{prop:dualF}), we obtain the result.
The case $\alpha<\beta$ can be proved in a similar way. 

Next we prove the converse.
By Proposition \ref{prop:Ponto},
 if $b\in F^{\infty}_{\frac{\alpha+\beta}4,-\rho-\eta}$ then
there exists $\varphi\in L^\infty_{0,-\rho-\eta}$ such 
that $P_{\frac{\alpha+\beta}2}(\varphi)=b$ and 
$\|\varphi\|_{L^\infty_{0,-\rho-\eta}}\simeq\|b\|_{F^{\infty}
_{\frac{\alpha+\beta}4,-\rho-\eta}}$.
Therefore 
$
\Lambda(f,g)=
\langle fg,b\rangle_{\frac{\alpha+\beta}2}
=\langle fg,\varphi\rangle_{\frac{\alpha+\beta}2}
$, for $f,g\in E$. 
Hence Fubini's theorem and H\"older's inequality give 
\[
|\Lambda(f,g)|
\le
\|\varphi\|_{L^\infty_{0,-\rho-\eta}}
\|f\|_{F^p_{\alpha,\rho}}\|g\|_{F^{p'}_{\beta,\eta}}.
\]
So if we consider the form 
$\widetilde \Lambda:
L^p_{\alpha,\rho}\times L^{p'}_{\beta,\eta} \to\C$ 
defined by 
$\widetilde \Lambda(f,g)
=\langle fg,\varphi\rangle_\alpha$ 
we have $\widetilde \Lambda=\Lambda$ on $E\times E$ 
and 
\[
\|\varphi\|_{L^\infty_{0,-\rho-\eta}}\simeq\|b\|_{F^{\infty}
_{\frac{\alpha+\beta}4,-\rho-\eta}}
\simeq\|\Lambda\|\le \|\widetilde\Lambda\|\le
\|\varphi\|_{L^\infty_{0,-\rho-\eta}}.
\]

\subsection{Proof of Corollary \ref{cor:weakfact}}
First we consider the case $1<p<\infty$.

By H\"older's inequality  it is clear that   
$F^p_{\alpha,\rho}\odot F^{p'}_{\beta,\eta}
\hookrightarrow 
F^1_{\alpha+\beta,\rho+\eta}$. Since $E$ is dense in 
both spaces, in order to prove that they coincide 
it is enough to prove that 
$\|h\|_{F^p_{\alpha,\rho}\odot F^{p'}_{\beta,\eta}}\simeq 
\|h\|_{F^1_{\alpha+\beta,\rho+\eta}}$, 
 for $h\in E$.

It is easy to check that the dual of 
$F^p_{\alpha,\rho}\odot F^{p'}_{\beta,\eta}$ is 
isometrically isomorphic  to the space 
of bounded Hankel bilinear forms 
on $F^p_{\alpha,\rho}\times F^{p'}_{\beta,\eta}$, 
which we denote by  $\mathcal{H}$. 
Namely, any $\Psi\in (F^p_{\alpha,\rho}\odot 
F^{p'}_{\beta,\eta})'$ defines a bounded bilinear form on 
$F^p_{\alpha,\rho}\times F^{p'}_{\beta,\eta}$ by 
$\Lambda(f,g)=\Psi(fg)$, $f,g\in E$, which satisfies 
$\|\Lambda\|=\|\Psi\|$.
Conversely, each $\Lambda\in\mathcal{H}$  defines a  
form $\Psi$ on 
$F^p_{\alpha,\rho}\odot F^{p'}_{\beta,\eta}$ by 
$\Psi(\sum_jf_j\,g_j)= \sum_j\Lambda(f_j,g_j)$ and 
$\|\Psi\|=\|\Lambda\|$.
By Theorem \ref{thm:bounded-bf}, the map 
$b\mapsto \Lambda_b=\langle\cdot,
b\rangle_{\frac{\alpha+\beta}2}$ 
is a topological isomorphism from  
$F^\infty_{\frac{\alpha+\beta}4,-\rho-\eta}$ onto 
$\mathcal{H}$. Therefore the duality 
$(F^1_{\alpha+\beta,\rho+\eta})'=
F^\infty_{\frac{\alpha+\beta}4,-\rho-\eta}$ with respect to 
the $\frac{\alpha+\beta}2$-pairing 
(see Proposition  \ref{prop:dualF}) gives
\[
\|h\|_{F^p_{\alpha,\rho}\odot\, F^{p'}_{\beta,\eta}}
=\sup_{\|\Psi\|=1}\,|\Psi(h)|
\simeq
\sup_{\|b\|_{F^\infty_{\frac{\alpha+\beta}4,-\varrho-\eta}}=1} |\langle 
h,b\rangle_{\frac{\alpha+\beta}2}|\simeq 
\|h\|_{F^1_{\alpha+\beta,\rho+\eta}}.
\]

The proof of the case $p=1$ is similar. 
It is clear that 
$F^1_{\alpha,\rho}\odot F^{\infty}_{\beta,\eta}
\hookrightarrow  F^1_{\alpha+\beta,\rho+\eta}$.
By Proposition  \ref{prop:dualF} we have $(\mathfrak{f}^\infty_{\alpha+\beta,\rho+\eta})'=
F^1_{\frac{\alpha+\beta}4,-\rho-\eta}$ with respect to 
the $\frac{\alpha+\beta}2$-pairing.
 Hence, arguing  as above, we have 
$\|h\|_{F^1_{\alpha,\rho}\odot\, 
\mathfrak{f}^{\infty}_{\beta,\eta}}
\simeq 
\|h\|_{F^1_{\alpha+\beta,\rho+\eta}}$
and 
		\[
		F^1_{\alpha+\beta,\rho+\eta}=
		F^{1}_{\alpha,\rho}\odot \mathfrak{f}^{\infty}_{\beta,\eta}
\hookrightarrow
	F^{1}_{\alpha,\rho}\odot F^{\infty}_{\beta,\eta}
\hookrightarrow F^1_{\alpha+\beta,\rho+\eta},
		\]
		which ends the proof.

\section{Proof of Theorem \ref{thm:hankel}}
\label{sec:proof13}

We begin  observing that by dilation we can reduce the 
proof of Theorem \ref{thm:hankel} to the case 
$\alpha=1$.
As usual, we denote 
$S_p(F^2_{\alpha,\rho},\overline{F^2_{\alpha,\rho}})$ 
by $S_p(F^2_{\alpha,\rho})$.

By Lemma \ref{lem:pairing}\eqref{item:pairing1}, 
the dilation operator 
$\Psi_\alpha(f)(z):=
f(\alpha^{\frac{-1}{2\ell}}z)$ is a 
topological isomorphism from 
$L^{p}_{\tau\alpha,\rho}$ onto $L^{p}_{\tau,\rho}$, 
$1\le p\le\infty$, $\tau>0$,  such that 
 $\Psi_\alpha(F^{p}_{\tau\alpha,\rho})=
F^{p}_{\tau,\rho}$ and 
 $\Psi_\alpha(\overline{F^{p}_{\tau \alpha,\rho}})=
\overline{F^{p}_{\tau,\rho}}$. 
Moreover, for $f\in E$, 
\[
\Psi_\alpha(\mathfrak{h}_{b,\alpha}(f))(z) =
\langle 
K_\alpha(\cdot,\alpha^{\frac{-1}{2\ell}}z)\,f,b 
\rangle_\alpha
=\langle 
K_\alpha(\alpha^{\frac{-1}{2\ell}}\cdot,z)\,f,b 
\rangle_\alpha.
\]
Therefore  
Lemma \ref{lem:pairing}\eqref{item:pairing2}
and \eqref{eqn:Kdelta} give 
\begin{align*}
\Psi_\alpha (\mathfrak{h}_{b,\alpha}(f))(z)
&=\alpha^{\frac{-n}{\ell}}\langle 
K_\alpha(\alpha^{\frac{-1}{\ell}}\cdot,z)\,\Psi_\alpha(f),
\Psi_\alpha(b) \rangle_1\\
&=\langle K_1(\cdot,z)\,\Psi_\alpha(f),
\Psi_\alpha(b) \rangle_1\\
&=\mathfrak{h}_{\Psi_\alpha(b),1}(\Psi_\alpha(f))(z).
\end{align*}
So the boundedness (compactness)  of the operator 
$\mathfrak{h}_{b,\alpha}$ on $F^p_{\alpha,\rho}$ is 
equivalent to the boundedness 
(respectively, compactness) of 
$\mathfrak{h}_{\Psi_\alpha(b),1}$ on $F^p_{1,\rho}$ and
\[
\|\mathfrak{h}_{b,\alpha} \|_{F^p_{\alpha,\rho}}
\simeq \|\mathfrak{h}_{\Psi_\alpha(b),1} 
\|_{F^p_{1,\rho}}.
\]

Similarly, $\mathfrak{h}_{b,\alpha}\in 
S_p(F^2_{\alpha,\rho})$ 
if and only if 
$\mathfrak{h}_{\Psi_\alpha(b),1} \in 
 S_p(F^2_{1,\rho})$, 
with equivalent norms
(see, for instance, \cite[Theorem 7.8]{Weidmann}).
Moreover, 
\[
\|\mathfrak{h}_{\Psi_\alpha(b),1} \|_{F^p_{1,\rho}} \simeq 
\|\Psi_\alpha(b) \|_{F^\infty_{\frac 12}}
\Longleftrightarrow 
\|\mathfrak{h}_{b,\alpha}\|_{F^\infty_{\frac \alpha 2}}
\simeq \|b\|_{F^\infty_{\frac \alpha 2}},
\]
and
\[
\|\mathfrak{h}_{\Psi_\alpha(b),1} 
\|_{S_p(F^2_{1,\rho})}
\simeq 
\|\Psi_\alpha(b) 
\|_{F^p_{\frac 12,\rho+2n(\ell-1)/p}}
\Longleftrightarrow 
\|\mathfrak{h}_{b,\alpha}\|
_{S_p(F^2_{\alpha,\rho})}
\simeq \|b\|_{F^p_{\frac \alpha 2,\rho+2n(\ell-1)/p}}.
\]

Hence from now on we only consider the case  
$\alpha=1$ and we will simplify the notations by writing
$\langle\cdot,\cdot\rangle$,  $\mathfrak{h_b}$, $K$, $P$, 
$\dots$, instead of $\langle\cdot,\cdot\rangle_1$,  
$\mathfrak{h_{b,1}}$, $K_1$, $P_1$, $\dots$

In order to prove Theorem \ref{thm:hankel}, we will use  
\eqref{eqn:decompK21} with $\gamma=1$ and 
$\alpha=\beta=1$, that is,
\begin{equation}\label{eqn:decomhankel}
K(w,z)=\sum_{k=0}^{n-1} G_{k,1,\frac 12}(w\,\overline z)
H_{k,1,\frac 12}(w\,\overline z)
+R_{n,1,\frac 12}(w\,\overline z).
\end{equation}
According to the choice in the statement of Theorem  
\ref{thm:decompK2}, we write  
 $G_k=G_{k,1,1/2}$, $H_k=H_{k,1,1/2}$, $k=0,\cdots, n-1$, $G_n=R_{n,1,1/2}$ and $H_n=1$.

Let 
$b\in F^\infty_{\beta}$, $0<\beta <2$. Since 
$\overline{b(z)}
=\langle K(\cdot,z),b\rangle$, \eqref{eqn:decomhankel} 
shows that 
\begin{equation}\label{eqn:repreb}
\overline{b(z)}
=\sum_{k=0}^n \langle H_{k}(\cdot\overline{z}),
\overline{\mathfrak{h}_b( 
G_{k}(\cdot\overline{z}))}\rangle
=\sum_{k=0}^n\langle G_{k}(\cdot\overline{z}),
\overline{\mathfrak{h}_b( 
H_{k}(\cdot\overline{z}))}\rangle.
\end{equation}
This  representation  formula is the main tool to prove Theorem \ref{thm:hankel}.

\subsection{
Proof of Theorem 
\ref{thm:hankel}\eqref{item:hankelbound}}\quad\par

In this section we prove that 
$\mathfrak{h}_{b}$ extends to a bounded 
(compact) operator from $F^p_{1,\rho}$ to 
$\overline{F^p_{1,\rho}}$ if and only if $b\in 
F^\infty_{\frac{1}2}$
 (respectively, $b\in \mathfrak{f}^\infty_{\frac 12}$), 
and in this case  
$\|\mathfrak{h}_{b}\|_{F^{p}_{1,\rho}}
\simeq \|b\|_{F^\infty_{\frac 12}}$.

\subsubsection{\textbf{Proof of the sufficient condition}} \quad\par

Assume $b\in F^\infty_{\frac 12}$. 
By Proposition \ref{prop:Ponto}, 
there exists $\varphi\in L^\infty$ such that 
$P(\varphi)=b$ and 
$\|\varphi\|_{L^\infty}\simeq\|b\|_{F^\infty_{\frac 12}}$. 
Therefore 
$
\mathfrak{h}_{b}(f)(z)=\langle f\,K(\cdot,z),\,b\rangle=
\langle f\,K(\cdot,z),\,\varphi\rangle
$,
 and consequently
\begin{equation}\label{eqn:necboundh}
|\mathfrak{h}_{b}(f)(z)|\le \|\varphi\|_{L^\infty}
\langle |f|,\,|K(\cdot,z)|\rangle.
\end{equation}
By Proposition \ref{prop:pqnormBergman},
\[
|\mathfrak{h}_{b}(f)(z)|
\le \|\varphi\|_{L^\infty}
\|f\|_{F^\infty_{1,\rho}}\,\|K(\cdot,z)\|_{F^1_{1,-\rho}}
\lesssim \|\varphi\|_{L^\infty}
\|f\|_{F^\infty_{1,\rho}}(1+|z|)^{-\rho}e^{\frac 12|z|^{2\ell}},
\]
so  
$\|\mathfrak{h}_{b,\alpha}\|_{F^{\infty}_{1,\rho}}
\lesssim \|b\|_{F^\infty_{\frac 12}}$.
Next, \eqref{eqn:necboundh}, Fubini's theorem and
Proposition \ref{prop:pqnormBergman} give
\begin{align*}
\|\mathfrak{h}_{b,\alpha}(f)\|_{F^{1}_{1,\rho}}
&\lesssim \|\varphi\|_{L^\infty}
\langle |f(w)|,\,\|K(\cdot,w)\|_{F^{1}_{1,\rho}}\rangle\\
&\lesssim \|\varphi\|_{L^\infty}
\langle |f(w)|,\,(1+|w|)^{\rho}e^{\frac 12|w|^{2\ell}}\rangle\\
&=\|\varphi\|_{L^\infty}\|f\|_{F^{1}_{1,\rho}},
\end{align*}
which proves that  
$\|\mathfrak{h}_{b,\alpha}\|_{F^{1}_{1,\rho}}
\lesssim \|b\|_{F^\infty_{\frac 12}}$.

By  Lemma \ref{lem:interpolation} we obtain  
$\|\mathfrak{h}_{b,\alpha}\|_{F^{p}_{1,\rho}}
\lesssim \|b\|_{F^\infty_{\frac 12}}$, for $1\le p\le\infty$.

Now assume that $b\in\mathfrak{f}^{\infty}_{\frac 12}$. 
Since $\mathfrak{f}^{\infty}_{\frac 12}$ is the
closure of the polynomials in $F^{\infty}_{\frac 12}$, there
is a sequence of polynomials $\{q_k\}_{k\in\N}$ such that
$\|q_k-b\|_{F^{\infty}_{\frac 12}}\to 0$.
Therefore
$\|\mathfrak{h}_{q_k}-\mathfrak{h}_{b}\|_
{F^{p}_{1,\rho}}\to 0$, because
\[
\|\mathfrak{h}_{q_k}
-\mathfrak{h}_{b}\|_{F^{p}_{1,\rho}}
=\|\mathfrak{h}_{q_k-b}\|_{F^{p}_{1,\rho}}
\lesssim \|q_k-b\|_{F^{\infty}_{\frac 12}}.
\]
Since
$\{\mathfrak{h}_{q_k}\}_{k\in \N}$ is a sequence of finite
rank operators, it follows that
$\mathfrak{h}_{b}:F^{p}_{1,\rho}
\to\overline{F^{p}_{1,\rho}}$
is compact.

\begin{rem}
\label{rem:hankelf}
Using the above arguments we have that if 
$b\in F^{\infty}_{\frac 12}$ then  $\mathfrak{h}_{b}$ is 
bounded on $\mathfrak{f}^{\infty}_{1,\rho}$. 
Indeed, by \eqref{eqn:necboundh} and Proposition 
\ref{prop:estimatesK},
\begin{align*}
|\mathfrak{h}_{b}(f)(z)|
&\lesssim
\langle \chi_R|f|,\,|K(\cdot,z)|\rangle
+
\langle (1-\chi_R)|f|,\,|K(\cdot,z)|\rangle\\
&\lesssim
e^{(R+1)^\ell|z|^\ell} \| f\|_{L^1_{2}}
+
\|(1-\chi_R)f\|_{L^{\infty}_{1,\rho}},\,
\|K(\cdot,z)\|_{F^{1}_{1,-\rho}}
\end{align*}
where $f\in \mathfrak{f}^{\infty}_{1,\rho}$ and $\chi_R$ 
denotes the characteristic function of the ball centered 
at $0$ and radius $R$. 
By Proposition \ref{prop:pqnormBergman} 
\[
(1+|z|)^\rho\,e^{-\frac 12|z|^{2\ell}}|\mathfrak{h}_{b}(f)(z)|
\lesssim 
\|f\|_{F^\infty_{1,\rho}}\,e^{-\frac 12|z|^{2\ell}
+(R+1)^\ell|z|^\ell}+\|(1-\chi_R)f\|_{L^{\infty}_{1,\rho}}.
\]
Since $f\in \mathfrak{f}^{\infty}_{1,\rho}$, 
$\|(1-\chi_R)f\|_{L^{\infty}_{1,\rho}}\to 0$ as $R\to\infty$.
Moreover, for any $R>0$,  
$(1+|z|)^\rho\,e^{-\frac 12|z|^{2\ell}
+(R+1)^\ell|z|^\ell}\to 0$ as $|z|\to\infty$.
That proves that 
$\mathfrak{h}_{b}(f)\in \mathfrak{f}^{\infty}_{1,\rho}$.
\end{rem}

\subsubsection{\textbf{Proof of the necessary  condition}}\quad\par

First we prove that if $\mathfrak{h}_b$ is bounded on 
$F^{p}_{1,\rho}$, then $b\in F^\infty_{\frac{1}2}$.

For $k=0,\cdots,n$, let us consider the "normalized" 
functions 
\begin{align*}
\widetilde{G}_{k,z}(w)&
:=(1+|z|)^{-\rho-(\ell-1)(2k+1-2n/p)} 
e^{-\frac{|z|^{2\ell}}8} G_{k}(w\overline z ),\\
\widetilde{H}_{k,z}(w)&:=(1+|z|)^{\rho-(\ell-1)(2n/p-2k-1)}
e^{-\frac{|z|^{2\ell}}8} H_{k}(w\overline z).
\end{align*}
By \eqref{claim:G21}-\eqref{claim:R21} we have 
$\|\widetilde{G}_{k,z}\|_{F^p_{1,\rho}}\lesssim 1$ and 
$\|\widetilde{H}_{k,z}\|_{F^{p'}_{1,-\rho}}\lesssim 1$.
Using the  representation  formula \eqref{eqn:repreb} 
we have 
\begin{equation}\label{eqn:necbounded}
e^{-\frac{|z|^{2\ell}}4}\, \overline{b(z)}
=\sum_{k=0}^n\langle \widetilde{H}_{k,z},
\overline{\mathfrak{h}_b( \widetilde{G}_{k,z})}\rangle,
\end{equation}
so, by Schwarz's  inequality,
$\|b\|_{F^\infty_{\frac 12}}
\lesssim \|\mathfrak{h}_b\|_{F^{p}_{1,\rho}}$.

Now assume that $\mathfrak{h}_b$ is compact on 
$F^{p}_{1,\rho}$, $1<p<\infty$. 
By \eqref{eqn:necbounded} we have 
\[
e^{-\frac{|z|^{2\ell}}4}|b(z)|\lesssim \sum_{k=0}^n
\bigl\|\mathfrak{h}_b( \widetilde 
G_{k,z})\bigr\|_{F^{p}_{1,\rho}}.
\]
Consequently, 
in order to show that $b\in \mathfrak{f}^\infty_{\frac 12}$,
 it is enough to prove that 
$\|\mathfrak{h}_b(\tilde G_{k,z})\|_{F^{p}_{1,\rho}}\to 0$
 as $|z|\to\infty$.
Since, for $|w|\le R$, 
$|G_k(w\overline z )|\lesssim e^{(R^\ell+1)|z|^\ell}$,
 $\widetilde G_{k,z}$ converges uniformly to $0$ on 
compact sets as $|z|\to\infty$.
This fact together with 
$\|\widetilde{G}_{k,z}\|_{F^p_{1,\rho}}\lesssim 1$  
easily shows that $\tilde G_{k,z}\to 0$  
weakly in $F^{p}_{1,\rho}$ as $|z|\to\infty$
(see, for instance, \cite[Lemma 5.1]{Ca-Fa-Pa-Pe}).
Therefore, the compactness of  $\mathfrak{h}_b$ 
implies that 
$\|\mathfrak{h}_b(\tilde G_{k,z})\|_{F^{p}_{1,\rho}}\to 0
$ as $|z|\to\infty$.

The same argument proves that if  $\mathfrak{h}_b$ 
is compact on $\mathfrak{f}^{\infty}_{1,\rho}$, 
then $b\in \mathfrak{f}^\infty_{\frac 12}$. 

Next we use this result to prove that if  $\mathfrak{h}_b$ 
is compact on $F^{1}_{1,\rho}$ then 
$b\in \mathfrak{f}^\infty_{\frac 12}$. 

If  $\mathfrak{h}_b$ is compact on $F^{1}_{1,\rho}$, 
then it is bounded, so  $b\in F^\infty_{\frac 12}$.
By Remark \ref{rem:hankelf} we have that 
 $\mathfrak{h}_b$ is 
bounded on $\mathfrak{f}^{\infty}_{1,-\rho}$. 
The duality 
$(\mathfrak{f}^\infty_{1,-\rho})'=F^1_{1,\rho}$ 
together with the fact that 
$\langle \overline{\mathfrak{h}_b(f)},g\rangle
=\langle f, \overline{\mathfrak{h}_b(g)}\rangle
=\langle fg,b\rangle$, give that $\mathfrak{h}_b$ 
 is compact in $F^{1}_{1,\rho}$  
if and only if it is compact in 
$\mathfrak{f}^{\infty}_{1,-\rho}$, which implies   
$b\in \mathfrak{f}^\infty_{\frac 12}$. 

Finally, if $\mathfrak{h}_b$ 
 is compact in $F^{\infty}_{1,\rho}$  then 
$b\in F^\infty_{\frac 12}$ and $\mathfrak{h}_b$ 
 is bounded on $F^{1}_{1,-\rho}$. 
So it  is compact on $F^{\infty}_{1,\rho}$  
if and only if it is compact on 
$F^1_{1,-\rho}$, which  implies   that 
$b\in \mathfrak{f}^\infty_{\frac 12}$.

\subsection{Proof of Theorem \ref{thm:hankel}\eqref{item:hankelschat}}\quad\par

In this section we prove that 
 $\mathfrak{h}_{b}$ is in 
$S_p(F^2_{1,\rho})=S_p(F^2_{1,\rho}, 
\overline{F^2_{1,\rho}})$ if and only if
$b\in F^p_{\frac 12,2n(\ell-1)/p}$, and, in this case, 
$\|\mathfrak{h}_{b}\|_{S_p(F^2_{1,\rho})}
\simeq \|b\|_{F^p_{\frac 12,2n(\ell-1)/p}}$.

We start this section recalling 
some well-known results concerning to the Schatten 
class $S_p(H_0,H_1)$, where $H_0$ and $H_1$ are 
separable complex   Hilbert  spaces. See, for instance, 
\cite[Chapter 7]{Weidmann}.

Let $T$ be a compact linear  operator from $H_0$ to 
$H_1$.  Then $|T|:=(T^*T)^{1/2}$ is a compact positive 
operator on $H_0$, so we may consider its sequence 
of eigenvalues  $\{s_k(T)\}_{k\in\N}$, which are usually 
called the singular values of $T$.

For $0<p<\infty$, the Schatten class $S_p(H_0,H_1)$ 
consists of all compact linear  operators $T$ from 
$H_0$ to $H_1$ such that 
$$
\|T\|_{S_p(H_0,H_1)}^p
:=\sum_{k=1}^{\infty} s_k(T)^p<\infty.
$$
Moreover, $S_\infty(H_0,H_1)$ is the space of all the 
bounded linear operators from $H_0$ to $H_1$.

Note that $(S_p(H_0,H_1), \|\cdot\|_{S_p(H_0,H_1)})$ 
is a Banach space for $p\ge 1$ and a quasi-Banach space for $p<1$.
Moreover, since $ \|T\|_{S_q(H_0,H_1)} \le \|T\|_{S_p(H_0,H_1)}$ for $p<q$ and $T\in S_p(H_0,H_1)$, we have the embedding 
\begin{equation*}
	S_p(H_0,H_1)\hookrightarrow S_q(H_0,H_1),\qquad (0<p<q\le\infty).
	\end{equation*}

By using the polar decomposition of $T$, it turns out that there exist two orthonormal systems $\{u_k\}_{k\in\N}$ and $\{v_k\}_{k\in\N}$ of $H_0$ 
and $H_1$, respectively, such that
\[
T(f)=\sum_{k=1}^{\infty} s_k(T)
\langle f,u_k\rangle_{_{H_0}}v_k.
\]

Note that if 
$T_k(f):=s_k(T)\langle f,u_k\rangle_{_{H_0}}v_k$, 
then $\|T_k\|_{S_p(H_0,H_1)}=s_k(T)$. 
So if $T\in S_1(H_0,H_1)$, then the rank one operators 
$T_k$ satisfy
\begin{equation}
	\label{eqn:TS1}
	\sum_{k=1}^n T_k\to T\mbox{ in $S_1(H_0,H_1)$}
	\,\,\mbox{ and }\,\,
	\Bigl\|\sum_{k=1}^n T_k\Bigr\|_{S_1(H_0,H_1)}
=\sum_{k=1}^n \|T_k\|_{S_1(H_0,H_1)}.
\end{equation}

We end this section by recalling the interpolation identity
\begin{equation}
\label{eqn:interpolationS}
\left(S_1(H_0,H_1),S_\infty(H_0,H_1)\right)_{[\theta]}
=S_{1/(1-\theta)}(H_0,H_1) \qquad(0<\theta<1).
\end{equation}
See, for instance, \cite[Theorem 2.6]{Zhu2007}.	

The following lemma  is easy to check.

\begin{lem}\label{lem:rankoneS}
$T:F^{2}_{1,\rho}\to\overline{F^{2}_{1,\rho}}$ 
is a bounded linear operator of rank one  
if and only if   there 
are non zero functions $g\in F^{2}_{1,-\rho}$ and  $h\in 
F^{2}_{1,\rho}$ such that $T(f)=\langle f,g 
\rangle\overline{h}$, for any $f\in F^{2}_{1,\rho}$.	
	Moreover,  in this case, $\|T\|_{S_p(F^{2}_{1,\rho})}
	\simeq\|g\|_{F^{2}_{1,-\rho}}\|h\|_{F^{2}_{1,\rho}}$, for 
$1\le p\le\infty$.
\end{lem}

\subsubsection{\textbf{Proof of the sufficient 
condition}} \quad\par

The sufficient condition is a direct consequence of the 
following result.

\begin{prop}\label{prop:Schatten-suf}
	For $1\le p\le\infty$, 
	the operator $b\mapsto \mathfrak{h}_{b}$ 
 is  bounded from $F^p_{\frac 12, \frac{2n(\ell-1)}p}$ to $S_p(F^{2}_{1,\rho})$.
\end{prop}

In order to prove Proposition \ref{prop:Schatten-suf},
 we will need the following  interpolation  Lemma. 

\begin{lem}
Let $1<p<\infty$. Then
\begin{align}
\label{item:interpolationL}
(L^{1}_{1/2,2n(\ell-1)}, L^{\infty}_{1/2})_{[1/p']}&=
L^{p}_{1/2,2n(\ell-1)/p},\qquad \text{and}\\
\label{item:interpolationF}
(F^{1}_{1/2,2n(\ell-1)}, F^{\infty}_{1/2})_{[1/p']}&=
F^{p}_{1/2,2n(\ell-1)/p},
\end{align}
\end{lem}

\begin{proof}
We begin with the proof of \eqref{item:interpolationL}.
Since $f\mapsto f(z)e^{-\frac{|z|^{2\ell}}2}$ is an 
isometric isomorphism from  $L^{p}_{1/2,2n(\ell-1)/p}$  
onto  $L^p((1+|z|)^{2n(\ell-1)}\,dV(z))$, Riesz-Thorin theorem gives \eqref{item:interpolationL}.

By Proposition \ref{prop:Ponto}, $P_{\frac 12}$ is 
bounded from $L^{p}_{1/2,2n(\ell-1)/p}$ to 
$F^{p}_{1/2,2n(\ell-1)/p}$ and it is the identity on 
$F^{p}_{1/2,2n(\ell-1)/p}\hookrightarrow 
L^{p}_{1/2,2n(\ell-1)/p}$. Thus  
$F^{p}_{1/2,2n(\ell-1)/p}$ is a retract of 
$L^{p}_{1/2,2n(\ell-1)/p}$, $1\le p\le\infty$ and, 
consequently, \eqref{item:interpolationF}
 follows from \eqref{item:interpolationL}.
\end{proof}

\begin{proof}[Proof of Proposition
\ref{prop:Schatten-suf}] 
By the interpolation identities 
\eqref{item:interpolationF} and 
\eqref{eqn:interpolationS}
 it is enough to prove the result for $p=1$ and $p=\infty$. 
Since the last case has been done  in the previous 
section, we only have to deal with the case $p=1$.

Assume 
$b\in F^1_{\frac 12, 2n(\ell-1)}$.  
By Corollary \ref{cor:pointwise}, 
$b\in F^\infty_{\frac 12}$ and $b=P_{\frac 12}b$. 
Therefore, for $f\in E$ we have
\begin{align*}
(\mathfrak{h}_{b}(f))\,(z)&
=\int_{\C^n} f(u)\, 
\overline{b(u)}\,K(u,z)\,e^{-|u|^{2\ell}}dV(u)\\
&=\int_{\C^n} f(u)
\int_{\C^n}\overline{b(w)}\,K_{\frac 12}(w,u)\, 
e^{-\frac{|w|^{2\ell}}2}dV(w)\,K(u,z)\,e^{-|u|^{2\ell}}dV(u),
\end{align*}
and  Fubini's theorem gives
\begin{equation}
\label{eqn:h:Bochner}
(\mathfrak{h}_{b}(f))\,(z)
=\int_{\C^n}\overline{b(w)}\,
(\mathfrak{h}_{K_{\frac 12}(\cdot,w)} f)(z)\, 
e^{-\frac{|w|^{2\ell}}2}dV(w).
\end{equation}

This allows us to consider  the  following Bochner 
integral
\begin{equation}
\label{eqn:operator:Bochner:integral}
\int_{\C}\overline{b(w)}\,
{\mathfrak h}_{K_{1/2}(\cdot,w)}\,
e^{-\frac{|w|^{2\ell}}2}dV(w).
\end{equation}

 By Bochner's integrability theorem  
(see for instance \cite[p. 133]{yosida}), the   
 $S_1(F^{2}_{1,\rho})$-convergence of 
the Bochner's 
integral~{\eqref{eqn:operator:Bochner:integral}}  
means that the integrand
\begin{equation*}
S(w):=\overline{b(w)}\, 
      {\mathfrak h}_{K_{1/2}(\cdot,w)}
\end{equation*}
is an $S_1(F^{2}_{1,\rho})$-valued strongly measurable 
function on $\C$ which satisfies
\begin{equation}\label{eq:Bochner:integrability2}
\int_{\C}\|S(w)\|_{S_1(F^{2}_{1,\rho})}
\,e^{-\frac{|w|^{2\ell}}2}dV(w)
<\infty.
\end{equation}

We are going to show that $S(w)$ is an operator of 
rank at most one, for every $w\in\C$, and next we estimate its $S_1(F^{2,\ell}_{1,\rho})$-norm.

For any $w\in\C$ and $f\in E$, we have 
\begin{equation}\label{eqn:hankelK121}
\mathfrak{h}_{K_{1/2}(\cdot,w)}(f)(z)
= 2^{-n/\ell}
   \langle f, K(\cdot,2^{-1/\ell} w)\rangle\,  
   K(2^{-1/\ell} w,z).
\end{equation}
Indeed, by \eqref{eqn:Kdelta},  
$K_{1/2}(\cdot,w)
=2^{-n/\ell} K(\cdot,2^{-1/\ell} w)$.
Therefore
	\begin{align*}
\mathfrak{h}_{K_{1/2}(\cdot,w)}(f)(z)
	&=2^{-n/\ell}
	\langle f K(\cdot,z),\,K(\cdot,2^{-1/\ell}w)
	\rangle\\
	&=2^{-n/\ell}  f(2^{-1/\ell}  w)K(2^{-1/\ell}  w,z)\\
	&=2^{-n/\ell}
	\langle f, K(\cdot,2^{-1/\ell} w)\rangle\,  
	K(2^{-1/\ell} w,z).
\end{align*}

So  $\mathfrak{h}_{K_{1/2}(\cdot,w)}$ is an operator 
of rank one and, 
by Lemma \ref{lem:rankoneS} and Proposition 
\ref{prop:pqnormBergman}, we obtain
\begin{equation}\label{eqn:hankelK122}\begin{split}
\|\mathfrak{h}_{K_{1/2}(\cdot,w)}\|_{S_1(F^{2}_{1,\rho})}
&\simeq \|K(\cdot,2^{-1/\ell} w)\|_{F^{2}_{1,-\rho}}
\|K(\cdot,2^{-1/\ell}  w)\|_{F^{2}_{1,\rho}}\\
&\simeq (1+|w|)^{2n(\ell-1)}\, e^{\frac{|w|^{2\ell}}4}.
\end{split}\end{equation}

Observe that \eqref{eqn:hankelK121} shows 
that $S$ is an $S_1(F^{2}_{1,\rho})$-valued function on 
$\C$. 
Moreover, it is $S_1(F^{2}_{1,\rho})$-strongly 
measurable because
\[
w\in\C\longmapsto\mathfrak{h}_{K_{1/2}(\cdot,w)}\in
S_1(F^{2}_{1,\rho})
\]
is continuous. That follows because 
$\mathfrak{h}_{K_{1/2}(\cdot,w)}
-\mathfrak{h}_{K_{1/2}(\cdot,v)}$ has rank at 
most~{$2$} and so
\begin{align*}
\|\mathfrak{h}_{K_{1/2}(\cdot,w)}
-\mathfrak{h}_{K_{1/2}(\cdot,v)}\|_{S_1(F^{2}_{1,\rho})}
 &\le 2\,
\|\mathfrak{h}_{\{K_{1/2}(\cdot,w)
-K_{1/2}(\cdot,v)\}}\|
_{S_{\infty}(F^{2}_{1,\rho})}\\
 &\stackrel{\mbox{\tiny$(1)$}}{\lesssim} \|K_{1/2}(\cdot,w)
            -K_{1/2}(\cdot,v)\|_{F^{\infty}_{1/2}}\\
 &\stackrel{\mbox{\tiny$(2)$}}{\lesssim} \|K_{1/2}(\cdot,w)
-K_{1/2}(\cdot,v)\|_{F^1_{\frac 12, 2n(\ell-1)}}
\stackrel{\mbox{\tiny$(3)$}}{\longrightarrow} 0,            
\end{align*}
as $w\to v$, where $(1)$, $(2)$ and $(3)$ are  
consequences of 
Theorem~{\ref{thm:hankel}\eqref{item:hankelbound}}, 
Corollary~{\ref{cor:pointwise}} and the dominated 
convergence theorem, respectively.

Now \eqref{eqn:hankelK122}
gives~{\eqref{eq:Bochner:integrability2}}:
\begin{align*}
\int_{\C}\|S(w)\|_{S_1(F^{2,\ell}_1)}\, 
e^{-\frac{|w|^{2\ell}}2}dV(w)
\lesssim \int_{\C}|b(w)|\,(1+|w|)^{2n(\ell-1)}
e^{-\frac{|w|^{2\ell}}4}\,dV(w).
\end{align*}
Therefore, by \eqref{eqn:h:Bochner},
${\mathfrak h}_{b}\in S_1(F^{2}_{1,\rho})$ and
$
\|{\mathfrak h}_{b}\|_{S_1(F^{2}_{1,\rho})}
\lesssim \|b\|_{F^{1}_{1/2,2n(\ell-1)}}.
$
\end{proof}

\subsubsection{\textbf{Proof of necessary condition}} \quad\par

The following definition is motived by \eqref{eqn:repreb}.
\begin{defn}
	For $T\in S_\infty(F^{2}_{1,\rho})$, let 
	\[
\Phi_T(z):=\sum_{k=0}^n\langle H_k(\cdot\,\overline{z}),
\overline{T\bigl(G_k(\cdot\,\overline{z})\bigr)}\, \rangle
\qquad(z\in\C).
	\]
\end{defn}	

Observe that $\Phi_{\mathfrak{h}_b}=\overline{b}$, 
by \eqref{eqn:repreb}. 
Therefore the necessary part in 
Theorem~{\ref{thm:hankel}}\eqref{item:hankelschat} is a 
direct consequence of the following proposition.

\begin{prop}
	For $1\le p\le\infty$, 
	the linear operator $T\mapsto \Phi_T$ is  bounded 
from $S_p(F^{2}_{1,\rho})$ to $L^{p}_{1/2,2n(\ell-1)/p}$.
\end{prop}

\begin{proof}
It is easy to check that $\Phi_T$ is a continuous function 
on $\C$. 
Indeed, if $z_j\to z$ in $\C$, estimates \eqref{claim:E1} 
and \eqref{claim:R1} and the dominated convergence 
theorem imply that
\[
H_k(\cdot\,\overline z_j)\to H_k(\cdot\,\overline z) 
\quad\mbox{in $F^{2}_{1,-\rho}$}
\quad\mbox{and}\qquad
G_k(\cdot\,\overline z_j)\to G_k(\cdot\,\overline z)
\quad\mbox{in $F^{2}_{1,\rho}$.}
 \]

So, taking into account the interpolation identities 
~{\eqref{eqn:interpolationS}} and 
\eqref{item:interpolationL},  it is enough to prove the 
proposition for $p=1$ and $p=\infty$.

The case $p=\infty$ follows from Schwarz inequality, 
the boundedness of $T$ and \eqref{eqn:decompK2}:
$$
|\Phi_T(z)|
\lesssim \|T\|_{S_\infty(F^{2}_{1,\rho})}\,
    \sum_{k=0}^n
 \|G_{k}(\cdot\overline{z})\|
_{F^2_{1,\rho}}
\|H_{k}(\cdot\overline{z})\|
_{F^{2}_{1,-\rho}}
\simeq  \|T\|_{S_\infty(F^{2}_{1,\rho})}
 e^{\frac{|z|^{2\ell}}{4}}.
$$
	
Now we prove the case $p=1$, that is, 
\begin{equation}\label{eq:operator:estimate}
\|\Phi_T\|_{L^{1}_{1/2,2n(\ell-1)}}
\lesssim\|T\|_{S_1(F^{2}_{1,\rho})}
\qquad(T\in S_1(F^{2}_{1,\rho})).
\end{equation}
By \eqref{eqn:TS1} we only have to 
prove~{\eqref{eq:operator:estimate}} 
for operators of rank one. So, taking into account  
Lemma \ref{lem:rankoneS}, we may assume that $T$ 
satisfies
\[
T(f)=\langle f,g\rangle\, \overline{h}
\qquad(f\in F^{2}_{1,\rho}), 
\]	
for some functions $g\in F^{2}_{1,-\rho}$ and 
$h\in F^{2}_{1,\rho}$.

	In this case,
\[
\Phi_T(z)=\sum_{k=0}^n
\langle G_k(\cdot\,\overline{z}),g\rangle\,
\langle H_k(\cdot\,\overline{z}),h\rangle,
\]
and Schwarz inequality gives
\[
\|\Phi_T\|_{L^1_{\frac 12,2n(\ell-1)}}\lesssim 
\sum_{k=0}^n I_k\,J_k,
\]
where
\begin{align*}
I_k^2&:=
\int_{\C^n}\left|\langle 
G_k(\cdot\overline{z}),g\rangle\right|^2\,
(1+|z|)^{-2\rho+2(\ell-1)(2n-2k-1)}e^{-\frac{|z|^{2\ell}}4}
dV(z)\\
J_k^2&:=
\int_{\C^n}\left|\langle 
H_k(\cdot\overline{z}),h\rangle\right|^2\,
(1+|z|)^{2\rho+2(\ell-1)(2k+1)}e^{-\frac{|z|^{2\ell}}4}
dV(z).
\end{align*}

Next we prove that $I_k\lesssim\|g\|_{F^{2}_{1,-\rho}}$ 
and  $J_k\lesssim \|h\|_{F^{2}_{1,\rho}}$, which, by  
Lemma \ref{lem:rankoneS},  give 
\[
\|\Phi_T\|_{L^1_{\frac 12,2n(\ell-1)}}\lesssim 
\|g\|_{F^{2}_{1,-\rho}}\,\|h\|_{F^{2}_{1,\rho}}
\simeq\|T\|_{S_1(F^{2}_{1,\rho})}.
\]

In order to prove the estimate $I_k\lesssim\|g\|_{F^{2}_{1,-\rho}}$, first note that 
Schwarz's inequality gives 
\begin{align*}
|\langle G_k(\cdot\overline{z}),g\rangle|^2
&\lesssim 
\int_{\C^n}|g(w)|^2|G_k(w\overline z)|
e^{-\frac{3|w|^{2\ell}}2}dV(w)\,
\int_{\C^n} |G_k(w\overline z)| 
e^{-\frac{|w|^{2\ell}}2}dV(w).
\end{align*}
Then, by \eqref{claim:E1} and \eqref{claim:R1}, we obtain
\[
|\langle G_k(\cdot\,\overline{z}),g\rangle|^2
\lesssim (1+|z|)^{(\ell-1)(2k+1-2n)}\, e^{\frac{|z|^{2\ell}}8}
\int_{\C^n}|g(w)|^2|G_k(\overline zw)|e^{-\frac{3|w|^{2\ell}}2}dV(w).
\]
Therefore Proposition \ref{prop:claims} with $\gamma=1$, $\alpha=\frac 14$ and $\theta=\frac 12$ 
gives 
\begin{align*}
I_k^2&\lesssim \int_{\C^n}|g(w)|^2
\|G_k(\cdot \overline w)\|
_{L^1_{\frac 14,-2\rho+(\ell-1)(2n-2k-1)}}\,
e^{-\frac{3|w|^{2\ell}}2}dV(w)\\
&\lesssim
\int_{\C^n}|g(w)|^2(1+|w|)^{-2\rho}
e^{-|w|^{2\ell}}dV(w)=\|g\|_{F^{2}_{1,-\rho}}.
\end{align*}

Similarly, replacing $\rho$ and $k$ by $-\rho$ and 
$n-1-k$, respectively, we obtain $J_k\lesssim 
\|h\|_{F^{2}_{1,\rho}}$.
\end{proof}

\end{document}